\pdfoutput=1 
\RequirePackage[l2tabu, orthodox]{nag} 
\documentclass[a4paper,reqno]{article} 
\usepackage[stretch=10]{microtype} 
\usepackage[german,french,UKenglish]{babel}
\usepackage[T1]{fontenc} 
\usepackage[utf8]{inputenc} 
\usepackage{csquotes} 
\usepackage[backend=biber,style=alphabetic,giveninits=true, 
isbn=false,sorting=nyt,maxnames=4,doi=false,url=false]{biblatex}
\usepackage{xpatch} 
\usepackage{letltxmacro} 
\usepackage{perpage} 
\usepackage{soul} 
\usepackage{xparse} 
\usepackage{amsmath}
\usepackage{amsfonts}
\usepackage{amssymb}
\usepackage{amsthm}
\usepackage{mathrsfs}
\usepackage{mathdots}
\usepackage{enumitem}
\usepackage{calc}
\usepackage{tocloft} 
\usepackage{titlesec} 
\usepackage{textcomp} 
\usepackage{tikz}
\usepackage{tikz-cd} 
\usetikzlibrary{calc}
\usetikzlibrary{matrix,arrows}
\usetikzlibrary{decorations.markings}

\usepackage{verbatim} 
\usepackage{aliascnt} 
\usepackage[colorlinks=true,urlcolor=blue!70!black,citecolor=blue!60!black,linkcolor=blue!60!black,pdfauthor={R. van Dobben de Bruyn},pdftitle={A variety that cannot be dominated by one that lifts}]{hyperref} 

\titlespacing\section{0pt}{12pt plus 4pt minus 2pt}{0pt plus 2pt minus 2pt}
\titlespacing\subsection{0pt}{12pt plus 4pt minus 2pt}{0pt plus 2pt minus 2pt}
\titlespacing\subsubsection{0pt}{0pt plus 2pt minus 2pt}{0pt plus 2pt minus 2pt}

\titleformat{\section}[block]{\Large\bfseries\scshape\filcenter}{\thesection.}{1ex}{}
\titleformat{\subsection}{\large\scshape\filcenter}{\thesubsection}{1ex}{}

\makeatletter

\ExplSyntaxOn
\cs_new:Npn \white_text:n #1
  {
    \llap{\textcolor{white}{\the\SOUL@syllable}\hspace{\fp_eval:n{0.01*#1} em}}
    \llap{\textcolor{white}{\the\SOUL@syllable}\hspace{-\fp_eval:n{0.01*#1} em}}
  }
\NewDocumentCommand{\whiten}{ m }
    {
      \int_step_function:nnnN {1}{1}{#1} \white_text:n
    }
\ExplSyntaxOff

\NewDocumentCommand{ \varul }{ D<>{5} O{0.2ex} O{0.1ex} +m } {%
\begingroup
\setul{#2}{#3}%
\def\SOUL@uleverysyllable{%
   \setbox0=\hbox{\the\SOUL@syllable}%
   \ifdim\dp0>\z@
      \SOUL@ulunderline{\phantom{\the\SOUL@syllable}}%
      \whiten{#1}%
      \llap{%
        \the\SOUL@syllable
        \SOUL@setkern\SOUL@charkern
      }%
   \else
       \SOUL@ulunderline{%
         \the\SOUL@syllable
         \SOUL@setkern\SOUL@charkern
       }%
   \fi}%
    \ul{#4}%
\endgroup
}

\makeatother

\makeatletter
\tikzcdset{
  open/.code     = {\tikzcdset{hook, circled};},
  closed/.code   = {\tikzcdset{hook, slashed};},
  open'/.code    = {\tikzcdset{hook', circled};},
  closed'/.code  = {\tikzcdset{hook', slashed};},
  circled/.code  = {\tikzcdset{markwith = {\draw (0,0) circle (.375ex);}};},
  slashed/.code  = {\tikzcdset{markwith = {\draw[-] (-.4ex,-.4ex) -- (.4ex,.4ex);}};},
  markwith/.code ={
    \pgfutil@ifundefined%
    {tikz@library@decorations.markings@loaded}%
    {\pgfutil@packageerror{tikz-cd}{You need to say %
      \string\usetikzlibrary{decorations.markings} to use arrows with markings}{}}{}%
    \pgfkeysalso{/tikz/postaction = {
      /tikz/decorate,
      /tikz/decoration={markings, mark = at position 0.5 with {#1}}}
    }
  },
}
\makeatother

\newtheoremstyle{thms}{1em}{0pt}{\itshape}{}{\bfseries}{.}{ }{}
\theoremstyle{thms}
\newtheorem{Thm}{Theorem}[section]				

\newaliascnt{Prop}{Thm}							
\newtheorem{Prop}[Prop]{Proposition}
\aliascntresetthe{Prop}

\newaliascnt{Lemma}{Thm}							
\newtheorem{Lemma}[Lemma]{Lemma}
\aliascntresetthe{Lemma}

\newaliascnt{Cor}{Thm}						
\newtheorem{Cor}[Cor]{Corollary}
\aliascntresetthe{Cor}

\newaliascnt{Conj}{Thm}							

\aliascntresetthe{Conj}

\newaliascnt{Question}{Thm}						
\newtheorem{Question}[Question]{Question}
\aliascntresetthe{Question}

\newtheoremstyle{defs}{1em}{0pt}{}{}{\bfseries}{.}{ }{}
\theoremstyle{defs}
\newaliascnt{Rmk}{Thm}							
\newtheorem{Rmk}[Rmk]{Remark}
\aliascntresetthe{Rmk}

\newaliascnt{Fact}{Thm}							

\aliascntresetthe{Fact}

\newaliascnt{Def}{Thm}							
\newtheorem{Def}[Def]{Definition}
\aliascntresetthe{Def}

\newaliascnt{Ex}{Thm}								
\newtheorem{Ex}[Ex]{Example}
\aliascntresetthe{Ex}

\newaliascnt{Con}{Thm}							
\newtheorem{Con}[Con]{Construction}
\aliascntresetthe{Con}

\newaliascnt{Not}{Thm}							

\aliascntresetthe{Not}

\newaliascnt{Setup}{Thm}							
\newtheorem{Setup}[Setup]{Setup}
\aliascntresetthe{Setup}

\newaliascnt{Picture}{Thm}						
\newtheorem{Picture}[Picture]{Picture}
\aliascntresetthe{Picture}

\theoremstyle{thms}
\newtheorem{thm}{Theorem}

\newtheorem*{thm*}{Theorem}
\newtheorem{question}{Question}

\newtheorem*{question*}{Question}
\newtheorem*{lemma*}{Lemma}

\LetLtxMacro\oldproof\proof						
\renewcommand{\proof}[1][Proof]{\oldproof[#1]\unskip} 

\newcommand{\boldref}[1]{\autoref{#1}}

\newenvironment{itemize*} 
  {\begin{itemize}
    \setlength{\itemsep}{1em}
    \setlength{\parskip}{-1em}
    \setlength{\topsep}{0pt}
    \setlength{\partopsep}{0pt}}
  {\end{itemize}}
  
\newenvironment{enumerate*}
  {\begin{enumerate}
    \setlength{\itemsep}{1em}
    \setlength{\parskip}{-1em}
    \setlength{\topsep}{0pt}
    \setlength{\partopsep}{0pt}}
  {\end{enumerate}}
  
\setlist{itemsep=0em,topsep=0cm,partopsep=0em,parsep=\lineskip}

\setlist[enumerate]{label=\normalfont(\arabic*)}

\setlist[itemize]{leftmargin=1.3em}

\newcommand{\CH}{\operatorname{CH}}
\newcommand{\Pic}{\operatorname{Pic}}
\newcommand{\PIC}{\operatorname{\textbf{\textup{Pic}}}}
\newcommand{\Alb}{\operatorname{Alb}}
\newcommand{\ALB}{\operatorname{\textbf{\textup{Alb}}}}
\newcommand{\Jac}{\operatorname{\textbf{\textup{Jac}}}}

\newcommand{\Gal}{\operatorname{Gal}}
\newcommand{\Hom}{\operatorname{Hom}}

\newcommand{\Aut}{\operatorname{Aut}}

\newcommand{\HOM}{\operatorname{\textbf{\textup{Hom}}}}
\newcommand{\End}{\operatorname{End}}

\newcommand{\Defo}{\operatorname{Def}}

\newcommand{\Mor}{\operatorname{Mor}}
\newcommand{\MOR}{\operatorname{\textbf{\textup{Mor}}}}

\newcommand{\NS}{\operatorname{NS}}
\newcommand{\Spec}{\operatorname{Spec}}
\newcommand{\SPEC}{\operatorname{\textbf{\textup{Spec}}}}
\newcommand{\Frob}{\operatorname{Frob}}

\newcommand{\Sch}{\operatorname{\underline{Sch}}}
\newcommand{\Set}{\operatorname{\underline{Set}}}

\newcommand{\ab}{\operatorname{ab}}

\newcommand{\Char}{\operatorname{char}}
\newcommand{\cts}{_{\operatorname{cts}}}
\let\oldtop\top
\newcommand{\T}{^\oldtop}
\renewcommand{\top}{^{\operatorname{top}}}

\newcommand{\et}{\operatorname{\acute et}}

\newcommand{\id}{\operatorname{id}}

\newcommand{\rA}{\longrightarrow}
\newcommand{\Ra}{\Rightarrow}

\newcommand{\La}{\Leftarrow}

\newcommand{\LRa}{\Leftrightarrow}

\newcommand{\punct}[1]{\makebox[0pt][l]{\,#1}} 

\newcommand{\N}{\mathbb N}
\newcommand{\Z}{\mathbb Z}
\newcommand{\Q}{\mathbb Q}
\newcommand{\R}{\mathbb R}
\newcommand{\C}{\mathbb C}
\newcommand{\F}{\mathbb F}
\renewcommand{\P}{\mathbb P}
\newcommand{\A}{\mathbb A}

\newcommand{\x}{^\times}

\DefineBibliographyStrings{UKenglish}{edition = {edition}} 
\AtEveryBibitem{\ifentrytype{book}{\clearfield{url}\clearfield{pages}}{}}
\renewbibmacro{in:}{}
\DeclareFieldFormat[book,incollection,inproceedings]{number}{\mkbibbold{#1}}
\DeclareFieldFormat[article]{volume}{\mkbibbold{#1}}
\DeclareFieldFormat*{title}{\mkbibitalic{#1}} 
\DeclareFieldFormat*{booktitle}{#1} 
\DeclareFieldFormat*{journaltitle}{#1}
\DeclareFieldFormat{postnote}{#1}
\DeclareFieldFormat{pages}{p. #1}
\DeclareFieldFormat[incollection,inproceedings]{volume}{.#1}
\DeclareFieldFormat[thesis]{type}{\bibstring{#1}\addcomma\addspace}
\DefineBibliographyStrings{english}{bibliography = {References},}
\AtEveryBibitem{\clearlist{publisher}\clearname{editor}\clearfield{location}}
\renewcommand*{\newunitpunct}{\addcomma\space}
\renewcommand*{\newblockpunct}{\addperiod\space}
\renewcommand*{\bibfont}{\footnotesize}

\renewcommand*{\mkbibnamefamily}[1]{\textsc{#1}} 
\renewcommand*{\mkbibnameprefix}[1]{\textsc{#1}}

\newbibmacro{titlelink}[1]{%
	\iffieldundef{doi}{%
		\iffieldundef{url}{%
			#1
		}{%
			\href{\thefield{url}}{#1}%
		}%
	}{%
		\href{https://dx.doi.org/\thefield{doi}}{#1}%
	}%
}
\DeclareFieldFormat{title}{\usebibmacro{titlelink}{\mkbibemph{#1}}}

\xpatchbibdriver{article}
  {\newunit
   \usebibmacro{note+pages}}
  {}
  {}{}
\renewbibmacro*{journal+issuetitle}{%
  \usebibmacro{journal}%
  \setunit*{\addspace}%
  \iffieldundef{series}
    {}
    {\newunit
     \printfield{series}%
     \setunit{\addspace}}%
  \usebibmacro{volume+number+eid}%
  \newunit
  \usebibmacro{note+pages}%
  \setunit{\addspace}%
  \usebibmacro{issue+date}%
  \setunit{\addcolon\space}%
  \usebibmacro{issue}%
  \newunit}
\xpatchbibdriver{inbook}
  {\newunit\newblock
  \usebibmacro{chapter+pages}}%
  {}%
  {}{}
\xpatchbibdriver{incollection}
  {\newunit\newblock
  \usebibmacro{chapter+pages}}%
  {}%
  {}{}
\xpatchbibdriver{inproceedings}
  {\newunit\newblock
  \usebibmacro{chapter+pages}}%
  {}%
  {}{}
\xapptobibmacro{maintitle+booktitle}
  {\printfield{pages}%
  \newunit}
  {}{}
\xpatchbibdriver{inproceedings}
  {\newunit\newblock
  \usebibmacro{series+number}}%
  {\newunit\newblock
  \usebibmacro{series+number+volume}}
  {}{}
\xpatchbibdriver{inproceedings}
  {\iffieldundef{maintitle}
    {\printfield{volume}%
     \printfield{part}}
    {}}
  {}
  {}{}
\newbibmacro*{series+number+volume}{%
  \printfield{series}%
  \setunit*{\addspace}%
  \printfield{number}%
  \printfield{volume}%
  \newunit}

\xpatchbibdriver{article}
  {\newunit\newblock
  \usebibmacro{in:}}%
  {\setunit{\addperiod\space}
  \usebibmacro{in:}}%
  {}{}
\xpatchbibdriver{book}
  {\newunit\newblock
  \usebibmacro{series+number}%
  \newunit\newblock
  \printfield{note}}%
  {\setunit{\addperiod\space}%
  \usebibmacro{series+number}
  \setunit{\addperiod\space}%
  \printfield{note}}%
  {}{}
\xpatchbibdriver{inbook}
  {\newunit\newblock
  \usebibmacro{maintitle+booktitle}}%
  {\setunit{\addperiod\space}
  \usebibmacro{maintitle+booktitle}}%
  {}{}
\xpatchbibdriver{incollection}
  {\newunit\newblock
  \usebibmacro{in:}}%
  {\setunit{\addperiod\space}}
  {}{}
\xpatchbibdriver{inproceedings}
  {\newunit\newblock
  \usebibmacro{in:}}%
  {\setunit{\addperiod\space}}%
  {}{}
\xpatchbibdriver{incollection}
  {\newunit\newblock
  \usebibmacro{series+number}}%
  {\setunit{\addperiod\space}
  \usebibmacro{series+number}}%
  {}{}
\xpatchbibdriver{thesis}
  {\newunit\newblock
  \printfield{type}}%
  {\setunit{\addperiod\space}
  \printfield{type}}%
  {}{}
\xpatchbibmacro{doi+eprint+url}
  {\newunit\newblock
  \iftoggle{bbx:eprint}}%
  {\setunit{\addperiod\space}
  \iftoggle{bbx:eprint}}%
  {}{}
\xpatchbibmacro{series+number}
  {\printfield{series}}
  {\iffieldundef{shortseries}
    {\printfield{series}}
    {\printfield{shortseries}}}
  {}{}
\xpatchbibmacro{addendum+pubstate}
  {\newunit\newblock
  \printfield{pubstate}}
  {\setunit{\addperiod\space}\newblock
  \printfield{pubstate}}
  {}{}

\begin{document}

\renewcommand{\sectionautorefname}{Section}
\renewcommand{\subsectionautorefname}{Subsection}
\setlength{\cftbeforetoctitleskip}{3em}
\renewcommand{\contentsname}{\hfill\Large\bfseries\scshape Contents\hfill}
\renewcommand{\cftaftertoctitle}{\hfill}

\renewcommand{\contentsname}{\hfill\Large\bfseries\scshape Contents\hfill}
\renewcommand{\cftaftertoctitle}{\hfill}

\begin{center}
\noindent\makebox[\linewidth]{\rule{13cm}{0.4pt}}
\vspace{-.4em}

{\LARGE{\textsc{\textbf{A variety that cannot be\\[.2em] dominated by one that lifts}}}}

\vspace{2.0em}

{\large{\textsc{Remy van Dobben de Bruyn}}}

\vspace{.5em}
\rule{7.5cm}{0.4pt}
\vspace{2.5em}
\end{center}

\renewcommand{\abstractname}{\small\bfseries\scshape Abstract}

\begin{abstract}\noindent
We prove a precise version of a theorem of Siu and Beauville on morphisms to higher genus curves, and use it to show that if a variety $X$ in characteristic $p$ lifts to characteristic $0$, then any morphism $X \to C$ to a curve of genus $g \geq 2$ can be lifted along. We use this to construct, for every prime $p$, a smooth projective surface $X$ over $\bar \F_p$ that cannot be rationally dominated by a smooth proper variety $Y$ that lifts to characteristic $0$.
\end{abstract}

\vspace{1em}

\phantomsection
\section*{Introduction}\label{Sec intro}
Given a smooth proper variety $X$ over a field $k$ of characteristic $p > 0$, a \emph{lift} of $X$ to characteristic $0$ consists of a DVR\footnote{One can also define lifts over a more general base, but this reduces to the case of a DVR at the expense of enlarging the residue field (see e.g.~\cite[Lem.~6.1.3]{Thesis}).} $R$ of characteristic $0$ with residue field $k$ and a flat proper $R$-scheme $\mathcal X$ whose special fibre $\mathcal X_0$ is isomorphic to $X$. 

Varieties that lift enjoy some of the properties of varieties in characteristic $0$. For example, for varieties of dimension $d \leq p$ that lift over the Witt ring $W(k)$, or even its characteristic $p^2$ quotient $W_2(k)$, the Hodge--de Rham spectral sequence degenerates and ample line bundles satisfy Kodaira vanishing \cite{DelIll}. If $X$ is a minimal surface of general type that lifts to characteristic $0$, then the Bogomolov--Miyaoka--Yau inequality holds \cite[Ex.~11.5]{LiedSurv}, and the same is true if $X$ lifts to $W_2(k)$ and $p > 2$ \cite[Thm.~13]{Langer}.

However, Serre showed \cite{SerLift} that not every smooth projective variety can be lifted to characteristic $0$. Serre's example is constructed as a quotient $Y \to X$ of a liftable variety by a finite group action. The following well-known open problem arises naturally from this construction:

\begin{question}\label{Q main}
Given a smooth proper variety $X$ over $\bar \F_p$, does there exist a smooth proper variety $Y$ and a surjection $Y \twoheadrightarrow X$ such that $Y$ lifts to characteristic $0$?
\end{question}

The main result of this paper is a negative answer to \boldref{Q main}:

\begin{thm}\label{thm intro}
Let $C$ be a supersingular curve over $\bar \F_p$ of genus $g \geq 2$, and $X \subseteq C^3$ a sufficiently general divisor. If $Y$ is a smooth proper variety admitting a dominant rational map $Y \dashrightarrow X$, then $Y$ cannot be lifted to characteristic $0$.
\end{thm}

Here, a curve $C$ is \emph{supersingular} if its Jacobian is a supersingular abelian variety, and by \emph{a sufficiently general divisor} we mean that there is a Zariski open $U \subseteq \NS(C^3) \otimes \Q$ such that for every very ample line bundle $\mathscr L$ whose N\'eron--Severi class lands in $U$, a general member $X \in |\mathscr L^{\otimes n}|$ for $n \gg 0$ satisfies the conclusion of the theorem. See \boldref{Thm main} and \boldref{Rmk Zariski open}.

There is a long history of improving the properties of a variety after a cover; classical examples are Chow's lemma and resolution of singularities. \boldref{Q main} asks whether one can ``resolve characteristic $p$ pathologies'' in the same way. 

According to Serre \cite[31 mars 1964, notes]{GrothendieckSerre}, Grothendieck had asked the more ambitious question whether every smooth projective variety is dominated by a product of curves. This was answered negatively by Serre \cite[31 mars 1964]{GrothendieckSerre}, and later independently by Schoen \cite{Schoen}. Deligne recently showed \cite{DeligneLetter} that \boldref{Q main} has a negative answer if one further assumes that $K(X) \to K(Y)$ is purely inseparable. In both cases, the examples constructed are surfaces, which is the smallest possible example as curves are unobstructed.

The class of liftable varieties plays a particularly important role in motivic and cohomological questions over finite fields. Lifting to characteristic $0$ gives access to powerful techniques that are unavailable in positive characteristic. For example, the recent proofs of the Tate conjecture for K3 surfaces over finitely generated fields \cite{ChaTate}, \cite{MP}, \cite[Appendix A]{KMP} proceed by lifting the K3 surface to characteristic 0 \cite{DelLift} and using the (transcendental) Kuga--Satake construction.

A positive answer to \boldref{Q main} would give a strategy for deducing cohomological statements in positive characteristic from the characteristic $0$ versions. Indeed, if $f \colon Y \to X$ is a surjective morphism of smooth proper varieties, then the pullback \mbox{$f^* \colon H^*(X) \to H^*(Y)$} for any Weil cohomology theory $H$ is injective \cite[Prop.~1.2.4]{KleDix}. If we can find such $Y$ that lifts to characteristic $0$, then one can try to deduce properties of $H^*(X)$ from the characteristic $0$ analogue. 

For example, if $X$ is a smooth projective variety over $\bar\F_p$ and $\alpha \in \CH^i(X)_\Q$ is an algebraic cycle, then it is expected\footnote{See also \cite{Independence} for the equivalence of this conjecture to other classical conjectures on independence of $\ell$ of \'etale cohomology of varieties over finite fields.} that the vanishing or nonvanishing of $\operatorname{cl}(\alpha) \in H^{2i}_{\et}(X,\Q_\ell)$ does not depend on the prime $\ell$. A strategy for this problem would be to dominate the \emph{pair} $(X,\alpha)$ by a pair $(Y,\beta)$ that can be lifted. The present paper shows that this is not even possible in absence of the cycle $\alpha$.

However, the Tate conjecture predicts that every \emph{motive} over $\bar\F_p$ embeds into one coming from a liftable variety. Indeed, under Tate, results from Honda \cite{Honda} and Tate \cite{Tate} imply that the category of Chow motives over $\bar\F_p$ is generated by (liftable) abelian varieties; see e.g.\ \cite[Rmk.\ 2.7]{MilFiniteFields}.

In this light, \boldref{Q main} can be seen as a direct approach to this (rather weak) consequence of the Tate conjecture. This also indicates that a purely motivic (cohomological) obstruction to \boldref{Q main} is unlikely.

It was already known that the surfaces constructed in \boldref{thm intro} cannot be dominated by a product of curves \cite[Prop.\ 7.2.1]{Schoen}, but as far as we know even non-liftability seems to be new.

For smooth proper surfaces liftability is a birational invariant, but Liedtke and Satriano showed that this fails for smooth projective threefolds, as well as for singular surfaces \cite{LiedSat}. Achinger and Zdanowicz constructed beautiful elementary examples of smooth projective \emph{rational} varieties that cannot be lifted to any ring in which $p \neq 0$ \cite{AchZda}. Their examples are in many ways as nice as possible, e.g.\ their cohomology is generated by algebraic cycles.

\phantomsection
\subsection*{Outline of the proof}\label{Sec outline}
Like in Serre's example, we have no direct obstruction to liftability of $X$ or $Y$. Rather, we prove that additional geometric structure can be lifted along, and then set up our example to obtain a contradiction. 

In Serre's argument, the additional structure that lifts is a finite \'etale Galois cover $X' \to X$. This structure lives \emph{above} $X$, so we have no way to use it on $Y$. Instead, we lift structure \emph{below} $X$:

\begin{thm}\label{thm lift morphism}
Let $X$ be a variety in characteristic $p$, and let $\mathcal X \to \Spec R$ be a lift over a DVR $R$. Let $\phi \colon X \to C$ be a morphism to a smooth projective curve of genus $g \geq 2$ such that $\phi_* \mathcal O_X = \mathcal O_C$. Then $\phi$ can be lifted to a morphism $\widetilde\phi \colon \mathcal X \to \mathcal C$, up to an extension of $R$ and a Frobenius twist of $C$.
\end{thm}

A precise version is given in \boldref{Thm lift morphism to curve}. The proof relies on a classification of morphisms $X \to C$ to higher genus curves that depends only on the fundamental group of $X$. For this, we need the following precise version of Siu--Beauville's theorem \cite[Thm.~4.7]{Siu}, \cite[Appendix]{Cat}.

\begin{thm}\label{thm SB}
Let $X$ be a smooth proper variety over an algebraically closed field $k$ of characteristic $0$, let $\ell$ be a prime, and let $g_0 \geq 2$. Then the association $\phi \mapsto \phi_*$ induces a bijection
\begin{align*}
\left\{\phi \colon X \twoheadrightarrow C\ \bigg|\ g(C) \geq g_0\right\}\raisebox{-1.0em}{$\!\!\!\diagup$}\raisebox{-1.3em}{$\!\!\sim$} &\rA \left\{\rho \colon \pi_1^{\et,\ell}(X) \to \Gamma_g^\ell \text{ open}\ \bigg|\ g \geq g_0\right\}\raisebox{-1.0em}{$\!\!\!\diagup$}\raisebox{-1.3em}{$\!\!\sim$}
\end{align*}
on equivalence classes for naturally defined equivalence relations. 
\end{thm}

Here, $\Gamma_g^\ell$ denotes the pro-$\ell$ fundamental group of a genus $g$ smooth projective curve. On the left hand side, two pairs $(C_1,\phi_1), (C_2,\phi_2)$ are equivalent if they both factor through a third pair $(C,\phi)$ (\boldref{Def equiv Mor}). On the right hand side, two open maps $\rho_1, \rho_2$ are equivalent if their \emph{abelianisations} both factor through the abelianisation of a third map $\rho$ (\boldref{Def equiv Hom}). 
The classical statement of Siu--Beauville is recalled in \boldref{Thm SB}, and our version is \boldref{Thm bijection}. The proof is a refinement of Beauville's argument \cite[Appendix]{Cat}.

To deduce \boldref{thm lift morphism} from \boldref{thm SB}, we use the specialisation isomorphism
\[
\operatorname{sp} \colon \pi_1^{\et,\ell}(X_{\bar K}) \stackrel\sim\to \pi_1^{\et,\ell}(\mathcal X_{\bar k}),
\]
where $k$ is the residue field and $K$ the fraction field of $R$. The map $\phi \colon X \to C$ gives rise to a map $\phi_* \colon \pi_1^{\et,\ell}(X_{\bar K}) \cong \pi_1^{\et,\ell}(X_{\bar k}) \twoheadrightarrow \Gamma_g^\ell$, which in characteristic $0$ comes from some morphism $\phi' \colon \mathcal X_{\bar K} \to C'$ to a higher genus curve. The proof is carried out by relating $(\phi',C')$ to the pair $(\phi,C)$ we started with.

With \boldref{thm lift morphism} in place, we want to study varieties admitting many morphisms to higher genus curves. We will work on a product $\prod_{i=1}^r C_i$ of curves of genera $g_i \geq 2$, and we define for each $i \in \{1,\ldots,r\}$ an obstruction $E_i(\mathscr L)$ for a line bundle on $\prod_i C_i$ to lift to $\prod_i \mathcal C_i$ for lifts $\mathcal C_i$ of the $C_i$. The definition and main properties of $E_i(\mathscr L)$ are given in \boldref{Sec line bundles}. The isomorphism
\[
\Pic\left(\prod_{i=1}^r C_i\right) \cong \prod_{i = 1}^r \Pic(C_i) \times \prod_{i < j} \Hom_k(J_i, J_j)
\]
suggests that we should look at supersingular curves $C_i$, because the supersingular abelian varieties $J_i = \Jac_{C_i}$ have more automorphisms than is possible in characteristic $0$. In \boldref{Sec Rosati} we construct a line bundle $\mathscr L$ on a power $C^3$ of a supersingular curve $C$ of genus $g \geq 2$ such that no multiple $\mathscr L^{\otimes n}$ for $n > 0$ can be lifted to $\prod_i \mathcal C_i$ for \emph{any} choice of lifts $\mathcal C_i$ of $C_i$; see \boldref{Lem generate End}.

The proof of \boldref{thm intro} then roughly goes as follows. Choose a line bundle $\mathscr L$ on $C^3 = \prod_i C_i$ as above, and let $X \in |\mathscr L|$ be a general member. If $Y \twoheadrightarrow X$ is a surjective morphism (for simplicity), then consider the projections $\phi_i \colon Y \to C_i$. 
By \boldref{thm lift morphism}, if $\mathcal Y$ is a lift of $Y$, then the $\phi_i$ can be lifted to maps $\widetilde\phi_i \colon \mathcal Y \to \mathcal C_i$ (for simplicity we ignore Stein factorisation and Frobenius twists). Then the image of the product map
\[
\widetilde\phi \colon \mathcal Y \to \prod_{i=1}^3 \mathcal C_i
\]
is a divisor whose special fibre is a multiple of the reduced divisor $X$. But then a power of $\mathcal O_{\prod C_i}(X) = \mathscr L$ lifts to $\prod_i \mathcal C_i$, contradicting the choice of $\mathscr L$.

There are some additional technical difficulties one runs into, coming from the fact that the morphisms $Y \to C_i$ do not lift on the nose. Rather, one has to take their Stein factorisation $Y \to C'_i \to C_i$ first, and then the morphisms $Y \to C'_i$ only lift up to a power of Frobenius $F \colon C'_i \to C''_i$ (see \boldref{thm lift morphism}). 

One therefore has to devise an argument that is flexible with respect to finite covers $C'_i \to C_i$. We facilitate this as follows:
\begin{itemize}
\item We show that the obstruction $E_i(\mathscr L)$ to the liftability of $\mathscr L$ to $\prod_i \mathcal C_i$ is well-behaved with respect to pullback under finite morphisms (\boldref{Lem pullback and composition}). This is the reason we use this intermediate obstruction, rather than working directly with nonliftable line bundles.
\item At the end of the argument, we take the scheme-theoretic image. This is only well-behaved with respect to pushforward, not pullback. Pullback and pushforward can be interchanged as long as the inverse image of $X$ under $\prod_i C'_i \to \prod_i C_i$ is still irreducible (\boldref{Lem irreducible inverse image}).
\item But we have to define $X$ \emph{before} we know what the finite covers $C'_i \to C_i$ are. We call a divisor $X \subseteq \prod_i C_i$ \emph{stably irreducible} if its inverse image in $\prod_i C'_i$ is irreducible, regardless of the covers $C'_i \to C_i$. A Bertini theorem proves that a general member of $|\mathscr L^{\otimes n}|$ for $n \gg 0$ satisfies this property (\boldref{Prop stably irreducible}).
\end{itemize}

\subsection*{Structure of the paper}
The paper is divided into three (roughly) equal parts, each spanning two sections:
\begin{itemize}
\item In \boldref{Sec SB} we prove \boldref{thm SB}, which we then use in \boldref{Sec lift morphism to curve} to prove \boldref{thm lift morphism}. This is the \emph{geometric} part of the argument.
\item In \boldref{Sec line bundles}, we study line bundles $\mathscr L$ on a product $\prod_i C_i$ and define an obstruction $E_i(\mathscr L)$ for $\mathscr L$ to lift. We use this in \boldref{Sec Rosati} to construct a line bundle on the third power $C^3$ of a supersingular curve that cannot be lifted. This is the \emph{cohomological} part of the argument.
\item In \boldref{Sec stably irreducible}, we construct stably irreducible divisors in $|\mathscr L^{\otimes n}|$ for $n \gg 0$. This gives the variety $X$ of \boldref{thm intro}. In \boldref{Sec main construction}, we carry out the construction and proof.
\end{itemize}
This paper presents the main result of the author's dissertation \cite{Thesis}. The statement and proof of \boldref{thm SB} are new; in \cite{Thesis} we use a more technical argument relying on results from nonabelian Hodge theory \cite[Thm.~10]{Simp}, \cite{CS} to deduce \boldref{thm lift morphism}. The dissertation further contains proofs of well-known results for which no detailed account in the literature was known to the author; we occasionally refer the reader there for extended discussion.

\numberwithin{equation}{section}
\phantomsection
\subsection*{Notation}\label{Sec notation}
A \emph{variety} over a field $k$ will mean a separated scheme of finite type over $k$ that is geometrically integral. When we say \emph{curve, surface, threefold}, etc., this is always understood to be a variety. 

A smooth proper variety $X$ over a perfect field $k$ of characteristic $p > 0$ is \emph{supersingular} if for all $i$, all Frobenius slopes on $H^i_{\operatorname{crys}}(X/W(k))[\tfrac{1}{p}]$ equal $\tfrac{i}{2}$. 
If $X$ is an abelian variety, this reduces to $i = 1$, hence it recovers the usual notion. If $X$ is a curve, then it is supersingular if and only if its Jacobian is.

For two $S$-schemes $X$ and $Y$, we will write $\Mor_S(X,Y)$ for the set of morphisms of $S$-schemes $X \to Y$, and $\MOR_S(X,Y)$ for the functor $\Sch_S \to \Set$ mapping $T \to S$ to $\Mor_T(X_T,Y_T)$ (or the scheme representing this functor, if it exists). Similarly, if $A$ and $B$ are abelian schemes over $S$, then $\Hom_S(A,B)$ will denote the group of homomorphisms $A \to B$ of abelian schemes, and $\HOM_S(A,B)$ will denote the group scheme of homomorphisms (see e.g.~\cite[Cor.~4.2.4]{Thesis} for representability). The group $\Hom_S(A,B) \otimes \Q$ will be denoted $\Hom^\circ_S(A,B)$.

If $\mathcal P$ is a property of schemes and $X \to S$ is a morphism of schemes, then $\mathcal P$ \emph{holds for a general fibre $X_s$} if there exists a dense open $U \subseteq S$ such that $\mathcal P(X_s)$ holds for all $s \in U$. We will sometimes omit mention of $\mathcal P$ and say that $X_s$ with $s \in U$ is a \emph{general member} of the family. If $S$ is a variety over a finite field $k$, then there need not exist a general member that is defined over $k$.

We will write $\pi_1^{\et}(X)$ for the \'etale fundamental group of a scheme $X$, $\pi_1^{\et,\ell}(X)$ for its maximal pro-$\ell$ quotient, and $\pi_1\top(X)$ for the topological fundamental group of a $\mathbb C$-variety $X$. All maps between profinite groups are assumed continuous. We will write $\Gamma_g$, $\widehat{\Gamma}_g$, and $\Gamma_g^\ell$ for the topological, \'etale, and pro-$\ell$ fundamental groups of a smooth projective genus $g$ curve over $\mathbb C$ respectively. We consistently ignore the choice of base point, because it does not affect the arguments.

A \emph{rng} is a ring without unit, and a \emph{$\Q$-rng} is a rng which is also a $\Q$-vector space. 

\subsection*{Acknowledgements}
First and foremost, I want to thank Johan de Jong for his support. This paper would not have been here without his endless generosity, as well as his encouragement to keep pushing, even and especially when the mathematics is not working out. I also thank Kiran Kedlaya, Frans Oort, Bhargav Bhatt, and Daniel Litt for encouraging me to think about the surface case. In particular, I am grateful for Bhatt's suggestion to traverse the same loop backwards in the argument of \boldref{Lem generate End}; this provided the improvement needed to go from a threefold to a surface (see also \boldref{Rmk Albert}). 

I thank Bhargav Bhatt, Raymond Cheng, Johan Commelin, David Hansen, Kiran Kedlaya, Raju Krishnamoorthy, Shizhang Li, Daniel Litt, Milan Lopuha\"a-Zwakenberg, Alena Pirutka, Jason Starr, and Burt Totaro for helpful discussions. I am grateful to Adrian Langer and Christian Liedtke for comments on an earlier version, and in particular to Raymond Cheng for extensive suggestions and corrections. Finally, I thank the defense committee for helpful suggestions.

\section{A precise version of Siu--Beauville}\label{Sec SB}
The following result was obtained independently by Siu \cite[Thm.~4.7]{Siu} and Beauville \cite[Appendix]{Cat}.

\begin{Thm}[Siu--Beauville]\label{Thm SB}
Let $X$ be a compact K\"ahler manifold, 
and let $g_0 \geq 2$. Then $X$ admits a surjection $\phi \colon X \to C$ to a compact Riemann surface $C$ of genus $g(C) \geq g_0$ if and only if there exists a surjection $\rho \colon \pi_1\top(X) \twoheadrightarrow \Gamma_{g_0}$.
\end{Thm}

We will upgrade this in \boldref{Thm bijection} to a bijection between suitable sets of surjections $X \to C$ and maps $\pi_1(X) \to \Gamma_g$ with finite index image. For our mixed characteristic application, it is convenient to work with the pro-$\ell$ fundamental group, but all arguments can also be carried out with the topological fundamental group (if $k = \C$) or the profinite fundamental group.

\begin{Def}\label{Def equiv Mor}
Write $\Mor(X,\geq\!\!g_0)$ for the set of pairs $(C,\phi)$ where $C$ is a smooth projective curve of genus $\geq g_0$ and $\phi \colon X \to C$ is a nonconstant morphism, up to isomorphism (as schemes under $X$). By de Franchis's theorem, this is a finite set if $g_0 \geq 2$ \cite{dF} (see e.g.~\cite{Mar} for a modern proof).

For $(C_1,\phi_1), (C_2,\phi_2) \in \Mor(X,\geq\!\!g_0)$, write $(C_1,\phi_1) \sim (C_2,\phi_2)$ if there exists $(C,\phi) \in \Mor(X,\geq\!\!g_0)$ such that both $\phi_i$ factor through $\phi$. This is equivalent to the statement that in the Stein factorisations
\[
X \stackrel{\phi'_i}\rA C_i' \stackrel{f_i}\rA C_i
\]
of the $\phi_i$, the pairs $(C_i',\phi'_i)$ agree in $\Mor(X,\geq\!\!g_0)$.
\end{Def}

\begin{Rmk}
If $f \colon X \to Y$ is a dominant morphism of normal varieties, then $f_* \colon \pi_1^{\et}(X) \to \pi_1^{\et}(Y)$ is an open homomorphism \cite[Lem.~11]{Kol} (see also \cite[Lem.~5.3.1]{Thesis} for a shorter and more general proof), which is surjective if $f$ has geometrically connected fibres. Recall that a continuous homomorphism 
of profinite groups is open if and only if its image has finite index.
\end{Rmk}

\begin{Rmk}
The group $H^1_{\et}(X, \Z_\ell)$ can be identified with the (continuous) group cohomology $H^1(\pi_1^{\et,\ell}(X),\Z_\ell) = \Hom\cts(\pi_1^{\et,\ell}(X),\Z_\ell)$, so any morphism $\rho \colon \pi_1^{\et,\ell}(X) \to \Gamma_g^\ell$ induces a pullback
\[
\rho^* \colon H^1(\Gamma_g^\ell,\Z_\ell) \to H^1_{\et}(X,\Z_\ell).
\]
\end{Rmk}

\begin{Def}\label{Def equiv Hom}
Write $\Hom^\circ(\pi_1^{\et,\ell}(X),\Gamma_{\geq g_0}^{\ell})$ for the set of open homomorphisms $\rho \colon \pi_1^{\et,\ell}(X) \to \Gamma_g^\ell$ for $g \geq g_0$, up to isomorphism (as groups under $\pi_1^{\et,\ell}(X)$). 
For elements $\rho_1, \rho_2 \in \Hom^\circ(\pi_1^{\et,\ell}(X),\Gamma_{\geq g_0}^\ell)$, we write $\rho_1 \sim \rho_2$ if there exists $\rho \in \Hom^\circ(\pi_1^{\et,\ell}(X),\Gamma_{\geq g_0}^\ell)$ such that the abelianisations $\rho_i^{\ab} \colon \pi_1^{\et,\ell}(X)^{\ab} \to \Gamma_{g_i}^{\ell,\ab}$ for $i \in \{1,2\}$ both factor through $\rho^{\ab}$. 
Equivalently, the pullbacks in $H^1_{\et}(X,\Z_\ell)$ satisfy
\[
\rho_i^* H^1(\Gamma_{g_i}^\ell,\Z_\ell) \subseteq \rho^* H^1(\Gamma_g^\ell,\Z_\ell),
\]
as subgroups of $H^1_{\et}(X,\Z_\ell)$, for $i \in \{1,2\}$. 
\end{Def}

\begin{Thm}\label{Thm bijection}
Let $X$ be a smooth proper variety over an algebraically closed field $k$ of characteristic $0$, let $\ell$ be a prime, and let $g_0 \geq 2$. Then the association $\phi \mapsto \phi_*$ induces a bijection
\[
\alpha \colon \Mor(X,\geq\!\!g_0)\raisebox{-.5em}{$\!\!\diagup$}\raisebox{-.8em}{$\!\!\sim$}\stackrel\sim\rA\ \!\Hom^\circ\!\left(\pi_1^{\et,\ell}(X),\Gamma_{\geq g_0}^\ell\right)\raisebox{-.8em}{$\!\!\!\diagup$}\raisebox{-1.1em}{$\!\!\sim$}
\]
on equivalence classes for the relations of \boldref{Def equiv Mor} and \boldref{Def equiv Hom}.
\end{Thm}

The proof will be given after \boldref{Cor S wedge}. By \boldref{Cor open subgroup Gamma} below, the classical Siu--Beauville theorem amounts to the statement that one side is nonempty if and only if the other is.

\begin{Rmk}\label{Rmk canonical representative}
In $\Mor(X,\geq\!\!g_0)$, every equivalence class $\mathcal C$ for $\sim$ has a canonical representative given by the unique $(C,\phi) \in \mathcal C$ such that $\phi_* \mathcal O_X = \mathcal O_C$. A priori, equivalence classes in $\Hom^\circ(\pi_1^{\et,\ell}(X),\Gamma_{\geq g_0}^\ell)$ do not have a preferred representative (see also \boldref{Rmk bijection alternative}). 
In fact, it is not even clear that $\sim$ defines an equivalence relation on $\Hom^\circ(\pi_1^{\et,\ell}(X),\Gamma_{\geq g_0}^\ell)$; this will follow from the proof.
\end{Rmk}

We first discuss some general properties of the groups $\Gamma_g^\ell$. The following result is an unstated consequence of \cite{And}.

\begin{Lemma}\label{Lem pro-ell exact}
Let $1 \to N \to G \to H \to 1$ be an exact sequence of finitely generated groups. If $H$ is an $\ell$-group and $N^\ell$ has trivial centre, then the pro-$\ell$ completion
\[
1 \to N^\ell \to G^\ell \to H^\ell \to 1
\]
is exact.
\end{Lemma}

\begin{proof}
By \cite[Prop.~3~and~Cor.~7]{And}, it suffices to show that the image of the conjugation action $G \to \Aut(N^{\ab}/\ell)$ is an $\ell$-group. But this map is trivial on $N$ since conjugation of $N$ acts trivially on $N^{\ab}$. Hence, it factors through $G/N = H$, which is an $\ell$-group.
\end{proof}

\begin{Cor}\label{Cor open subgroup Gamma}
Let $g \geq 2$, and let $U \subseteq \Gamma_g^\ell$ be an open subgroup of index $\ell^n$. Then $U \cong \Gamma_{\ell^n(g-1)+1}^\ell$.
\end{Cor}

The corresponding statement for $\Gamma_g$ (resp.\ $\widehat{\Gamma}_g$) follows from topology (resp.\ algebraic geometry). The difficulty is that pro-$\ell$ completion (resp.~maximal pro-$\ell$ quotient) is in general only right exact.

\begin{proof}
By solubility of $\ell$-groups, we may reduce to the case that $U$ is normal (or even $n = 1$). If $H = \Gamma_g^\ell/U$, then consider the surjection $\Gamma_g \twoheadrightarrow H$. Its kernel is $\Gamma_{\ell^n(g-1)+1}$, whose pro-$\ell$ completion has trivial centre \cite[Prop.~18]{And}. Therefore, \boldref{Lem pro-ell exact} implies that $U \cong \Gamma_{\ell^n(g-1)+1}^\ell$.
\end{proof}
\vskip-\lastskip
The following lemma addresses injectivity of the map of \boldref{Thm bijection}. It also holds in positive characteristic, which will be used in the proof of \boldref{Thm lift morphism to curve}.

\begin{Lemma}\label{Lem dominant to product}
Let $X$ be a smooth proper variety over an algebraically closed field $k$, and let $\ell$ be a prime invertible in $k$. Let $(C_i,\phi_i) \in \Mor(X,\geq\!2)$ for $i \in \{1,2\}$, and consider the following statements.
\begin{enumerate}
\item The $\phi_i$ satisfy $(C_1,\phi_1) \sim (C_2,\phi_2)$;\label{Item sim Mor}
\item The product morphism $\phi \colon X \to C_1 \times C_2$ is not dominant;\label{Item not dominant}
\item The pullbacks $\phi_i^* H^1_{\et}(C_i,\mathbb Z_\ell) \subseteq H^1_{\et}(X,\mathbb Z_\ell)$ have nonzero intersection.\label{Item H^1}
\end{enumerate}
Then the implications $\ref{Item sim Mor} \LRa \ref{Item not dominant} \La \ref{Item H^1}$ hold. If $\phi_{i,*} \mathcal O_X = \mathcal O_{C_i}$ for $i \in \{1,2\}$, then the reverse implication $\ref{Item not dominant} \Ra \ref{Item H^1}$ holds as well.
\end{Lemma}

\begin{proof}
Note that $\phi$ is not dominant if and only if its image is contained in a (possibly singular) curve in $C_1 \times C_2$, showing $\ref{Item sim Mor} \LRa \ref{Item not dominant}$. If $\phi \colon X \to C_1 \times C_2$ is dominant, then the pullback
\[
\phi^* \colon H^1_{\et}(C_1 \times C_2,\Z_\ell) \to H^1_{\et}(X,\Z_\ell)
\]
is injective \cite[Prop.~1.2.4]{KleDix}. Therefore, the pullbacks $\phi^*H^1_{\et}(C_i,\Z_\ell)$ are linearly disjoint, which proves $\ref{Item H^1} \Ra \ref{Item not dominant}$. Finally, if $\phi_{i,*} \mathcal O_X = \mathcal O_{C_i}$, then $(C_1,\phi_1) \sim (C_2,\phi_2)$ implies $(C_1,\phi_1) \cong (C_2,\phi_2)$, so $\ref{Item sim Mor} \Ra \ref{Item H^1}$ is clear.
\end{proof}
\vskip-\lastskip
The implication $\ref{Item sim Mor} \Ra \ref{Item H^1}$ is false in general: if $X = C \subseteq C_1 \times C_2$ is a smooth very ample divisor, then $H^1_{\et}(X,\Z_\ell) \cong H^1_{\et}(C_1,\Z_\ell) \oplus H^1_{\et}(C_2,\Z_\ell)$, so that the parts coming from $C_1$ and $C_2$ are linearly disjoint.

\begin{Def}\label{Def line bundle subgroup}
Let $X$ be a variety over an algebraically closed field $k$, and let $\ell$ be a prime number that is invertible in $k$. For a class $\eta \in H^1_{\et}(X,\mathbb Z_\ell)$ and $n \in \Z_{> 0}$, we write $\mathscr L_n(\eta)$ for the $\ell^n$-torsion line bundle given by
\[
\eta\!\! \pmod{\ell^n} \in H^1_{\et}(X,\mathbb Z/\ell^n) \cong H^1_{\et}(X,\mu_{\ell^n}),
\]
and we write $\langle \eta \rangle \subseteq \Pic^0(X)[\ell^{\infty}]$ for the group generated by $\mathscr L_n(\eta)$ for $n \in \mathbb Z_{>0}$. The set of $\eta \in H^1_{\et}(X,\Z_\ell)$ such that $H^1(X,\mathscr L) \neq 0$ for all $\mathscr L \in \langle \eta \rangle$ is denoted by $S = S^1_\ell$ (in analogy with the Green--Lazarsfeld loci $S^i \subseteq \PIC^0_X$ \cite{GL}).
\end{Def}

The key input of the proof of \boldref{Thm bijection} is the following $\ell$-adic version of Beauville's corollary \cite[Thm.~1]{BeauAnnulation} of the Green--Lazarsfeld generic vanishing theorem \cite{GL}. The proof relies on Beauville's result [\emph{loc.\ cit.}].

\begin{Prop}\label{Prop Beauville for H^1}
Let $X$ be a smooth projective variety over an algebraically closed field $k$ of characteristic $0$, and let $\eta \in H^1_{\et}(X,\mathbb Z_\ell)$ be a nonzero element. Then the following are equivalent:
\begin{enumerate}
\item There exists $(C,\phi) \in \Mor(X,\geq\!2)$ such that $\eta \in \phi^* H^1_{\et}(C,\mathbb Z_\ell)$;\label{Item phi}
\item There exists $\rho \in \Hom^\circ(\pi_1^{\et,\ell}(X),\Gamma_{\geq 2})$ such that $\eta \in \rho^* H^1(\Gamma_g^\ell,\Z_\ell)$;\label{Item rho}
\item For all $\mathscr L \in \langle \eta \rangle$, we have $H^1(X,\mathscr L) \neq 0$;\label{Item all L}
\item For infinitely many $\mathscr L \in \langle \eta \rangle$, we have $H^1(X,\mathscr L) \neq 0$.\label{Item many L}
\end{enumerate}
\end{Prop}


\begin{proof}
If $\eta \in \phi^* H^1_{\et}(C,\mathbb Z_\ell)$ for some surjection $\phi \colon X \to C$, then the easy direction  of \cite[Thm.~1]{BeauAnnulation} gives $H^1(X,\mathscr L) \neq 0$ for all $\mathscr L \in \langle \eta \rangle$. This shows $\ref{Item phi} \Ra \ref{Item all L}$. Implications $\ref{Item phi} \Ra \ref{Item rho}$ and $\ref{Item all L} \Ra \ref{Item many L}$ are obvious.

If $H^1(X,\mathscr L_i) \neq 0$ for infinitely many $\mathscr L_i \in \langle \eta \rangle$, then by \cite[Thm.~1]{BeauAnnulation} there exist morphisms $\phi_i \colon X \to C_i$ with $g(C_i) \geq 2$ such that $\mathscr L_i \in \phi_i^* \Pic^0(C_i)$. Since there are only finitely many possible $C_i$, one of them must occur infinitely many times, which forces $\langle \eta \rangle \subseteq \phi_i^* \Pic^0(C_i)$ since $\phi_i^*$ is a group homomorphism. This immediately implies $\eta \in \phi_i^* H^1_{\et}(C_i,\mathbb Z_\ell)$, proving $\ref{Item many L} \Ra \ref{Item phi}$.

Finally, assume $\eta \in \rho^*H^1(\Gamma_g^\ell,\Z_\ell)$ for some $\rho \in \Hom^\circ(\pi_1^{\et,\ell}(X),\Gamma_g)$ with $g \geq 2$. By \boldref{Cor open subgroup Gamma}, we may assume $\rho$ is surjective (increasing $g$ if necessary). Let $\tau \colon \Gamma_g^\ell \to \mathbb Z_\ell$ be the homomorphism such that $\eta = \rho^*(\tau) \in H^1_{\et}(X,\mathbb Z_\ell)$. If $\tau = 0$, then clearly $\eta \in S$, so we may assume $\tau$ is surjective. Then the surjection $\pi_1^{\et,\ell}(X) \twoheadrightarrow \mathbb Z^\ell \twoheadrightarrow \mathbb Z/\ell^n$ corresponds to the cyclic $\mathbb Z/\ell^n$-cover
\[
\pi_n \colon X_n = \SPEC_X\left(\bigoplus_{i=0}^{\ell^n-1} \mathscr L_n(\eta)^{\otimes i}\right) \to X,
\]
where $\mathscr L_n(\eta)$ is as in \boldref{Def line bundle subgroup}. In particular, we find
\begin{equation}\label{Eq pushforward X_n}
H^1(X_n,\mathcal O_{X_n}) = H^1(X,\pi_{n,*} \mathcal O_{X_n}) \cong \bigoplus_{i=0}^{\ell^n-1} H^1\!\left(X,\mathscr L_n(\eta)^{\otimes i}\right).
\end{equation}
But $\pi_1^{\et}(X_n)$ surjects onto $\ker(\Gamma_g^\ell \to \Z/\ell^n)$, which is isomorphic to $\Gamma_{\ell^n(g-1)+1}^\ell$ by \boldref{Cor open subgroup Gamma}, so Hodge theory gives
\[
h^1(X_n,\mathcal O_{X_n}) \geq \ell^n(g-1)+1.
\]
Thus, by (\ref{Eq pushforward X_n}) there are infinitely many $n$ such that $H^1(X_n,\mathscr L_n(\eta)^{\otimes i}) \neq 0$ for some $i \in (\mathbb Z/\ell^n)^\times$, showing that the final implication $\ref{Item rho} \Ra \ref{Item many L}$.
\end{proof}

\begin{Cor}\label{Cor S wedge}
Let $X$ be a smooth projective variety over an algebraically closed field $k$ of characteristic $0$. Then the locus $S \subseteq H^1_{\et}(X,\mathbb Z_\ell)$ of \boldref{Def line bundle subgroup} is the finite union
\[
S = \bigvee_{\substack{\phi \colon X \twoheadrightarrow C\\\phi_*\mathcal O_X = \mathcal O_C}} \phi^* H^1_{\et}(C,\Z_\ell).
\]
of the linearly disjoint saturated $\mathbb Z_\ell$-submodules $\phi^*H^1_{\et}(C,\Z_\ell) \subseteq H^1_{\et}(X,\mathbb Z_\ell)$ for $\phi \colon X \to C$ a morphism to a curve $C$ of genus $g \geq 2$ satisfying $\phi_* \mathcal O_X = \mathcal O_C$.
\end{Cor}

\begin{proof}
By property \ref{Item phi} of \boldref{Prop Beauville for H^1}, every element $\eta \in S$ is contained in $\phi^* H^1_{\et}(C,\Z_\ell)$ for some surjection $\phi \colon X \to C$ to a smooth projective curve $C$ of genus $g \geq 2$. Taking Stein factorisation, we may assume $\phi_* \mathcal O_X = \mathcal O_C$. For different $C$, these spaces are pairwise linearly disjoint by \boldref{Lem dominant to product}. The saturatedness statement follows since property \ref{Item many L} of \boldref{Prop Beauville for H^1} for $a\eta$ ($a \in \mathbb Z_\ell\setminus\{0\}$) implies the same for $\eta$.
\end{proof}

\begin{proof}[Proof of \boldref{Thm bijection}]
It is well-known that rational maps to curves of genus $\geq 1$ extend (see e.g.~\cite[Cor.~4.1.4]{Thesis}). Moreover, $\pi_1$ is a birational invariant \cite[Exp.~X, Cor.~3.4]{SGA1}, hence both sides of the statement only depend on the birational isomorphism class of $X$. Then Chow's lemma \cite[Thm.~5.6.1]{EGA2} and resolution of singularities \cite{Hir} reduce us to the smooth projective case.

By \boldref{Prop Beauville for H^1} and \boldref{Cor S wedge}, the union of the pullbacks $\rho^*H^1(\Gamma_g^\ell,\Z_\ell)$ for $\rho \in \Hom^\circ(\pi_1^{\et,\ell}(X),\Gamma_{\geq g_0})$ is the wedge sum
\begin{equation}
S = \bigvee_{\substack{\phi \colon X \twoheadrightarrow C\\\phi_*\mathcal O_X = \mathcal O_C}} \phi^* H^1_{\et}(C,\Z_\ell).\label{Eq S wedge}
\end{equation}
This defines a map
\[
\beta \colon \Hom^\circ\!\left(\pi_1^{\et,\ell}(X),\Gamma_{\geq g_0}^\ell\right) \rA \Mor(X,\geq\!\!g_0)\raisebox{-.5em}{$\!\!\diagup$}\raisebox{-.8em}{$\!\!\sim$},
\]
taking $\rho$ to the unique $(C,\phi)$ with $\phi_* \mathcal O_X = \mathcal O_C$ corresponding to the component of the wedge sum (\ref{Eq S wedge}) in which $\rho^* H^1(\Gamma_g^\ell,\Z_\ell)$ lands. Moreover, the fibres of $\beta$ are exactly the equivalence classes of $\sim$, showing that $\sim$ is an equivalence relation. Then $\beta$ descends to a two-sided inverse of $\alpha$.
\end{proof}

\begin{Rmk}\label{Rmk bijection alternative}
The proof shows that the surjections $\rho \colon \pi_1^{\et,\ell}(X) \twoheadrightarrow \Gamma_g^\ell$ for which $\rho^* H^1(\Gamma_g^\ell,\Z_\ell)$ is inclusionwise maximal correspond to the maximal linear subspaces of the set $S$ of \boldref{Prop Beauville for H^1}. One can use this to state \boldref{Thm bijection} in terms of maximal elements instead of equivalence classes.

However, there may be multiple surjections $\rho \colon \pi_1^{\et,\ell}(X) \twoheadrightarrow \Gamma_g^\ell$ for which the pullbacks $\rho^* H^1(\Gamma_g^\ell,\Z_\ell)$ form the same maximal subspace of $S$, so they are only maximal in a weak sense. For this reason, we chose to state \boldref{Thm bijection} in terms of equivalence classes.
\end{Rmk}

\begin{Cor}\label{Cor homotopy invariant}
Let $X$ be a smooth proper variety over an algebraically closed field $k$ of characteristic $0$, and let $g \geq 2$. Then the set of (isomorphism classes of) surjections $\phi \colon X \to C$ with $\phi_* \mathcal O_X = \mathcal O_C$ to a smooth projective curve $C$ of genus $g$ only depends on $\pi_1^{\et,\ell}(X)$. In particular, if $k = \mathbb C$, it is a homotopy invariant of $X$.
\end{Cor}

\begin{proof}
Apply \boldref{Thm bijection} to $g_0 = g$ and $g_0 = g+1$.
\end{proof}
\vskip-\lastskip
In contrast, the original Siu--Beauville theorem (\boldref{Thm SB}) only addresses whether or not there exists a morphism $X \to C$ to a curve of some fixed genus $g \geq 2$, not how many there are.

\begin{Rmk}\label{Rmk Catanese}
A similar result was obtained by Catanese \cite[Thm.~2.25]{Cat}, which deals more generally with \emph{Albanese general type fibrations} $\phi \colon X \to Y$, i.e.~maps to K\"ahler manifolds $Y$ with $q(Y) > \dim(Y)$ such that the image of $\operatorname{alb} \colon Y \to \Alb(Y)$ has dimension $\dim(Y)$. Catanese's characterisation uses certain real subspaces of $H^1(X,\mathbb C)$ instead of the fundamental group. As such, it does not generalise well to other algebraically closed fields of characteristic $0$.

Catanese also shows that if $\phi \colon X \to C$ is a morphism from a smooth proper scheme $X$ to a smooth projective curve $C$ of genus $g \geq 2$ over a field $k$ of characteristic $0$, then the forgetful transformation
\begin{align*}
\Defo_{(\phi \colon X \to C)} &\rA \Defo_C
\end{align*}
is an isomorphism of deformation functors \cite[Rmk.~4.10]{Cat}. The proof was rediscovered by the present author \cite[Thm.~5.5.1]{Thesis}. It relies on a Kodaira type vanishing theorem of Koll\'ar \cite[Thm.~2.1]{KolVanishing}; there are various reasons the required vanishing theorem fails in positive characteristic \cite[Rmk.~4.12]{Cat}, \cite[Ex.~5.5.4]{Thesis}. See also \boldref{Q deformation invariant} below.
\end{Rmk}

\section{Lifting morphisms to curves}\label{Sec lift morphism to curve}
We apply \boldref{Thm bijection} to prove that a morphism $\phi \colon X \to C$ of smooth proper varieties to a curve $C$ of genus $\geq 2$ lifts along with any lift of $X$.

\begin{Thm}\label{Thm lift morphism to curve}
Let $R$ be a DVR of characteristic $0$ with fraction field $K$ and algebraically closed residue field $k$. Let $\mathcal X \to \Spec R$ be a smooth proper morphism, let $C$ be a smooth projective curve over $k$ of genus $g \geq 2$, and let $\phi \colon \mathcal X_0 \to C$ be a morphism with $\phi_* \mathcal O_{\mathcal X_0} = \mathcal O_C$. Then there exists a generically finite extension $R \to R'$ of DVRs, a smooth proper curve $\mathcal C'$ over $R'$, a morphism $\phi' \colon \mathcal X_{R'} \to \mathcal C'$, and a commutative diagram
\begin{equation*}
\begin{tikzcd}[column sep=.6em]
 & \mathcal X_0 \ar{ld}[swap]{\phi}\ar{rd}{\phi_0'} & \\
C \ar{rr}[swap]{F} & & \mathcal C_0'\punct{,}
\end{tikzcd}
\end{equation*}
where $F$ is purely inseparable. In particular, $F$ is a power of the relative Frobenius if $\Char k = p > 0$, and $F$ is an isomorphism if $\Char k = 0$.
\end{Thm}

\begin{Rmk}
That is, if $X$ can be lifted, then so can any morphism $\phi \colon X \to C$ with $\phi_* \mathcal O_X = \mathcal O_C$ to a curve $C$ of genus $g \geq 2$, up to a generically finite extension $R \to R'$ and a purely inseparable morphism $C \to \mathcal C_0' = C^{(p^{-n})}$.
\end{Rmk}

\begin{Rmk}
If $\Char k = 0$, i.e.~$R$ is a DVR of equicharacteristic $0$, then we may in fact choose $R = R'$ by Catanese's deformation theoretic result (\boldref{Rmk Catanese}). We do not know if the extension $R \to R'$ and the purely inseparable morphism $C \to \mathcal C_0'$ are actually needed in mixed characteristic, nor whether a variant of the result is true in pure characteristic $p > 0$. In fact, as far as we know the following is still open (see also \cite[Rmk.~4.12]{Cat}):
\end{Rmk}

\begin{Question}\label{Q deformation invariant}
Let $X$ and $X'$ be deformation equivalent smooth proper varieties over an algebraically closed field $k$ of characteristic $p > 0$. If $X$ admits a dominant morphism to a curve of genus $g \geq 2$, then does $X'$ as well?
\end{Question}

\begin{proof}[Proof of \boldref{Thm lift morphism to curve}]
The morphism $\phi \colon \mathcal X_0 \to C$ induces a surjection
\[
\rho \colon \pi_1^{\et,\ell}(\mathcal X_0) \twoheadrightarrow \Gamma_g^\ell.
\]
By \cite[Exp.~X,~Cor.~3.9]{SGA1}, we have an isomorphism $\pi_1^{\et,\ell}(\mathcal X_0) \cong \pi_1^{\et,\ell}(\mathcal X_{\bar K})$, hence we may view $\rho$ as a map $\pi_1^{\et,\ell}(\mathcal X_{\bar K}) \twoheadrightarrow \Gamma_g^\ell$. By \boldref{Prop Beauville for H^1} and \boldref{Cor S wedge} there exists a unique morphism $\phi' \colon \mathcal X_{\bar K} \to \bar C'$ with $\phi_*' \mathcal O_{\mathcal X_{\bar K}} = \mathcal O_{\bar C'}$ to a smooth projective curve $\bar C'$ over $\bar K$ of genus $g(\bar C') \geq 2$ such that
\begin{equation}
\rho^*H^1(\Gamma_g^\ell,\Z_\ell) \subseteq \phi'^* H^1_{\et}(\bar C',\Z_\ell).\label{Eq inclusion}
\end{equation}
There exists a finite extension $K'$ of $K$ and a smooth projective curve $C'$ over $K'$ such that $C'_{\bar K'} \cong \bar C'$. Extending $K'$ if necessary, we may assume that $C'$ has a rational point and that the morphism $\phi' \colon \mathcal X_{\bar K} \to \bar C'$ is defined over $K'$. We then have a $\Gal(\bar K/K')$-equivariant surjection
\begin{equation}\label{Eq surjection pi_1}
\phi_*' \colon \pi_1^{\et,\ell}(\mathcal X_{\bar K}) \twoheadrightarrow \pi_1^{\et,\ell}(\bar C').
\end{equation}
Let $R'$ be the localisation of the integral closure of $R$ in $K'$ at any prime above $\mathfrak m_R$. Then the $\Gal(\bar K/K')$-action on $\pi_1^{\et,\ell}(\mathcal X_{\bar K})$ is unramified since $\mathcal X_{K'}$ has good reduction, hence by the surjection (\ref{Eq surjection pi_1}) the same is true for the $\Gal(\bar K/K')$-action on $\pi_1^{\et,\ell}(\bar C')$.
By Takayuki Oda's ``N\'eron--Ogg--Shafarevich for curves'' \cite[Thm.~3.2]{Oda}\footnote{Oda's paper only states the result over a number field, but the methods work over any DVR. See e.g.~\cite[Thm.~0.8]{Tam} for a proof over an arbitrary DVR.}, this implies that $C'$ has good reduction. Thus, there exists a smooth proper curve $\mathcal C' \to \Spec R'$ with generic fibre $C'$. 

Since $g(C') \geq 1$, the N\'eron mapping property \cite[Cor.~4.4.4]{BLR} for the abelian scheme $\Alb_{\mathcal C'/R'}$ implies that the morphism $\phi' \colon \mathcal X_{K'} \to C'$ extends uniquely to a morphism $\phi' \colon \mathcal X_{R'} \to \mathcal C'$. Since formation of $\phi_*'$ commutes with flat base change, we get $\phi_*' \mathcal O_{\mathcal X_{K'}} = \mathcal O_{\mathcal C'}$, hence $\phi'$ has geometrically connected fibres \cite[Cor.~4.3.2]{EGA3I}. Now consider the Stein factorisation of its special fibre $\phi_0'$:
\[
\mathcal X_0 \stackrel f\rA \widetilde C' \stackrel F\rA \mathcal C_0'.
\]
Since $\phi_0'$ has geometrically connected fibres, the finite part $F$ is radicial. Since $\mathcal C_0'$ and $\widetilde C'$ are smooth projective curves over an algebraically closed field, this implies $F \colon \mathcal C_0'^{(p^n)} \to \mathcal C_0'$ is a power $\Frob^n$ of the relative Frobenius. Finally, the specialisation isomorphism \cite[Exp.~X,~Cor.~3.9]{SGA1} and topological invariance of the \'etale site \cite[Exp.~IX,~Thm.~4.10]{SGA1} give an isomorphism
\[
\pi_1^{\et,\ell}(\mathcal C_{\bar K}') \cong \pi_1^{\et,\ell}(\widetilde C').
\]
Under this isomorphism and the comparison $\pi_1^{\et,\ell}(\mathcal X_{\bar K}) \cong \pi_1^{\et,\ell}(\mathcal X_0)$, the maps $f_* \colon \pi_1^{\et,\ell}(\mathcal X_0) \to \pi_1^{\et,\ell}(\widetilde C')$ and $\phi_*' \colon \pi_1^{\et,\ell}(\mathcal X_{\bar K}) \to \pi_1^{\et,\ell}(\mathcal C_{\bar K}')$ agree. Translating (\ref{Eq inclusion}) to this notation gives
\[
\phi^* H^1_{\et}(C,\Z_\ell) \subseteq f^*H^1_{\et}(\widetilde C',\Z_\ell),
\]
so \boldref{Lem dominant to product} $\ref{Item H^1} \Ra \ref{Item sim Mor}$ forces $(\phi,C) \cong (f,\widetilde C')$.
\end{proof}

\section{Line bundles on products of curves}\label{Sec line bundles}
We will give a criterion for a line bundle on a product $\prod_i C_i$ of curves in characteristic $p > 0$ that implies it cannot be lifted to $\prod_i \mathcal C_i$ for \emph{any} lifts $\mathcal C_i$ of the curves $C_i$; see \boldref{Prop no multiple lifts} below. We will give an example of this obstruction in \boldref{Sec Rosati}. The main definitions are given in \boldref{Def components of line bundle} and \boldref{Def corresponds to isogeny factor}. It is based on the following lemma.

\begin{Lemma}\label{Lem Pic product general}
Let $S$ be a scheme, and let $X_i \to S$ for $i \in \{1,\ldots,r\}$ be flat proper morphisms of finite presentation for which the Picard functor $\PIC_{X_i/S}$ and the Albanese $\ALB_{X_i/S}$ are representable (as scheme or algebraic space). Write $X \to S$ for the fibre product $X_1 \times_S \ldots \times_S X_r$. Then any choice of sections $\sigma_i$ of $f_i$ induces an isomorphism
\[
\PIC_{X/S} \cong \prod_{i = 1}^r \PIC_{X_i/S} \times \prod_{i < j} \HOM_S(\ALB_{X_i/S}, \PIC^0_{X_j/S}).
\] 
\end{Lemma}
\vskip-\lastskip
See \cite[TDTE~VI,~Thm~3.3(iii)]{FGA} for the definition and main existence theorem of the Albanese. We include a sketch of the proof of this well-known lemma because we want to refer to the argument later. A detailed discussion can be found in \cite[\S 4.4]{Thesis}.

\begin{proof}[Proof of Lemma (sketch).]
When $r = 2$ the sections $\sigma_i$ induce a section
\begin{align*}
\PIC_{X/S} &\to \PIC_{X_1/S} \underset S\times \PIC_{X_2/S} \\
\mathscr L&\mapsto (\sigma_2^*\mathscr L,\sigma_1^*\mathscr L)
\end{align*}
to the natural external tensor product map. The kernel consists of line bundles $\mathscr L$ on $X$ trivial along the coordinate axes $\sigma_1 \times X_2$ and $X_1 \times \sigma_2$. The trivialisation along $X_1 \times \sigma_2$ (viewed as a rigidificator for the Picard functor $\PIC_{X_2/S}$) shows that this data corresponds to a morphism  
\[
\phi \colon X_1 \to \PIC_{X_2/S}.
\]
The trivialisation along $\sigma_1 \times X_2$ shows that $\phi(\sigma_1) = 0$, hence $\phi$ lands inside $\PIC^0_{X_2/S}$. The Albanese property shows that $\phi$ factors uniquely through
\[
\ALB_{X_1/S} \to \PIC^0_{X_2/S}.
\]
This proves the result for $r = 2$, and the general case follows by induction.
\end{proof}

\begin{Rmk}\label{Rmk change i and j}
The choice to use $\HOM(\ALB_i,\PIC^0_j) = \HOM(\ALB_{X_i/S},\PIC^0_{X_j/S})$ instead of the version $\HOM(\ALB_j,\PIC^0_i)$ with $i$ and $j$ swapped is arbitrary. If we use the same sections $\sigma_i$, then replacing $\HOM(\ALB_i,\PIC^0_j)$ by $\HOM(\ALB_j,\PIC^0_i)$ takes the map $\phi \colon \ALB_i \to \PIC^0_j$ to its transpose $\phi\T \colon \ALB_j \to \PIC^0_i$.

Indeed, by \boldref{Lem Pic product general} applied to both $X_i \times_S X_j$ and $\ALB_i \times_S \ALB_j$, the Albanese map $X_i \times_S X_j \to \ALB_i \times_S \ALB_j$ induces an isomorphism on the $\HOM$ factor of the lemma (line bundles trivialised along a coordinate cross). Hence, we may reduce to the case of abelian schemes, where it follows from the definition of the transpose.
\end{Rmk}

\begin{Rmk}\label{Rmk change sections}
The choice of sections $\sigma_i$ of $X_i \to S$ does not affect the projection $\PIC_{X/S} \to \prod_{i < j} \HOM(\ALB_{X_i/S},\PIC^0_{X_j/S})$. Indeed, we may reduce to the case $r = 2$. Then the map $X_i \to \PIC^0_{X_j/S}$ is given by $x_i \mapsto \mathscr L_{x_i \times_S X_j} \otimes \mathscr L_{\sigma_i \times_S X_j}^{-1}$, which visibly does not depend on $\sigma_j$.

For dependence on $\sigma_i$, use \boldref{Rmk change i and j} to swap the roles of $i$ and $j$. We can also argue directly: changing $\sigma_i$ gives maps $\ALB_{X_i/S} \to \PIC^0_{X_j/S}$ that differ by at most a translation, so they have to agree since they are morphisms of abelian varieties. (See also \cite[Lem.~4.4.5]{Thesis} for an alternative point of view and additional details.)
\end{Rmk}

For the rest of this section, we will work in the following setup.

\begin{Setup}\label{Setup Picard group of a product of curves}
Let $k$ be a field, let $r \in \Z_{>0}$, and let $C_1,\ldots,C_r$ be smooth projective curves over $k$ with $C_i(k) \neq \varnothing$. Let $X = \prod_i C_i$ be their product. The principal polarisation from the theta divisor induces an isomorphism $\PIC_{C_i/k}^0 \cong \ALB_{C_i/k}$, and we will denote both by $J_i$.
\end{Setup}

\begin{Cor}\label{Cor Pic product curves}
The choice of rational points $c_i \in C_i(k)$ induces an isomorphism
\[
\Pic(X) \cong \prod_{i = 1}^r \Pic(C_i) \times \prod_{i < j} \Hom_k(J_i, J_j).\]
The projection $\Pic(X) \to \prod_{i < j} \Hom_k(J_i, J_j)$ does not depend on the choice of rational points $c_i \in C_i(k)$.
\end{Cor}

\begin{proof}
Immediate from \boldref{Lem Pic product general} and \boldref{Rmk change sections}.
\end{proof}

\begin{Def}\label{Def components of line bundle}
Given a line bundle $\mathscr L \in \Pic(X)$, we write
\[
\mathscr L = \left(\big(\mathscr L_i\big)_i, \big(\phi_{ji}\big)_{i < j}\right) \in \prod_{i=1}^r \Pic(C_i) \times \prod_{i < j} \Hom(J_i,J_j).
\]
For each $i \in \{1,\ldots,r\}$, write $E_i(\mathscr L) \subseteq \End^\circ(J_i)$ for the $\Q$-subrng (equivalently, $\Q$-subspace) generated by the compositions
\[
\phi_{i_1\ldots i_m} = \phi_{i_1i_2} \circ \ldots \circ \phi_{i_{m-1}i_m}
\]
for any $m \geq 2$ and $i_1,\ldots,i_m \in \{1,\ldots,r\}$ with $i_1 = i_m = i$. 
Here we write $\phi_{ij} = \phi_{ji}\T$ if $i < j$; see \boldref{Rmk change i and j}. By \boldref{Rmk change sections}, the $E_i(\mathscr L)$ do not depend on the choice of rational points $c_i \in C_i(k)$ (but the $\mathscr L_i$ do).
\end{Def}

\begin{Rmk}
The reason we only adjoin $\phi_{i_1\ldots i_m}$ for $m \geq 2$ and do not include the empty composition $\phi_{\varnothing} = 1$ is that $1$ is not preserved under pullback by finite morphisms $C'_i \to C_i$, so that \boldref{Lem pullback and composition} below would no longer be true.
\end{Rmk}

\begin{Picture}\label{Pic noncommutative}
The $\phi_{ji}$ sit in the following (\emph{non-commutative}) diagram, drawn when $r = 3$ and $r = 4$:
\begin{equation*}
\begin{tikzcd}[column sep=.5em]
 & J_1 \ar{ld}[swap]{\phi_{21}}\ar{rd}{\phi_{31}} & &\ &\ & J_1 \ar{rr}{\phi_{21}}\ar{d}[swap]{\phi_{31}}\ar{rrd} &\ & J_2 \ar{d}{\phi_{42}}\ar{lld} \\
J_2 \ar{rr}[swap]{\phi_{32}} & & J_3\punct{,} & & & J_3 \ar{rr}[swap]{\phi_{43}} & & J_4\punct{.}
\end{tikzcd}
\end{equation*}
The compositions $\phi_{i_1\ldots i_m}$ with $i_1 = i_m = i$ correspond to loops based at $i$, where an arrow travelled in reverse direction introduces a transpose $(-)\T$.
\end{Picture}

The $E_i(\mathscr L)$ are introduced because they behave well with respect to pullback under finite covers $C'_i \to C_i$ (\boldref{Lem pullback and composition}), as well as with respect to specialisation (\boldref{Lem specialisation and composition}). In particular, the $E_i(\mathscr L)$ provide an obstruction for a line bundle $\mathscr L$ on $\prod C_i$ to lift to $\prod \mathcal C_i$ for \emph{any} lifts  $\mathcal C_i$ of $C_i$ (\boldref{Prop no multiple lifts}).


\begin{Def}\label{Def corresponds to isogeny factor}
Let $C_i$ and $X$ be as in \boldref{Setup Picard group of a product of curves}, and let $\mathscr L \in \Pic(X)$. Then $\mathscr L$ \emph{corresponds to an isogeny factor $A$ of $J_i$} if there exists an isogeny factor
\begin{equation}
\begin{tikzcd}[every arrow/.append style={shift left}]
J_i \ar{r}{\pi} & A \ar{l}{\iota}
\end{tikzcd}\label{Dia isogeny factor}
\end{equation}
such that $E_i(\mathscr L) = \iota \End^\circ(A) \pi$. Here, $\pi$ is a surjective homomorphism and $\iota$ is an element of $\Hom^\circ(A,J_i)$ such that $\pi \iota = \id$. 

Equivalently, $E_i(\mathscr L) = p \End^\circ(J_i)p$ for some idempotent $p \in \End^\circ(J_i)$. Indeed, isogeny factors as in (\ref{Dia isogeny factor}) correspond to idempotents $p \in \End^\circ(J_i)$ by setting $p = \iota \pi$, and under this correspondence we have
$$\iota \End^\circ(A) \pi = p \End^\circ(J_i) p.$$
If $E_i(\mathscr L) = \End^\circ(J_i)$, then we say that $\mathscr L$ \emph{generates all endomorphisms of $J_i$}. This is a special case of the above, where we take $A = J_i$, or equivalently $p = \id$.
\end{Def}


\begin{Lemma}\label{Lem pullback and composition}
Let $C_i$ and $X$ be as in \boldref{Setup Picard group of a product of curves}, and let $C'_i$ and $X'$ satisfy the same assumptions. For each $i$, let $f_i \colon C'_i \to C_i$ be a finite morphism, and denote their product by $f \colon X' \to X$. Let $\mathscr L \in \Pic(X)$, and let $\mathscr L' = f^* \mathscr L$. Then for all $i$, we have
\[
E_i(\mathscr L') = f_i^*E_i(\mathscr L)f_{i,*}.
\]
If $\mathscr L$ corresponds to an isogeny factor $A$ of $J_i$, then $\mathscr L'$ corresponds to the isogeny factor $A$ of $J'_i$. The converse holds if $g(C'_i) = g(C_i)$.
\end{Lemma}

\begin{proof}
Let $c'_i \in C'_i(k)$ be rational points, and let $c_i \in C_i(k)$ be their images. Let $\mathscr L_i$ ($\mathscr L'_i$) and $\phi_{ji}$ ($\phi'_{ji}$) denote the components of $\mathscr L$ ($\mathscr L'$) as in \boldref{Def components of line bundle}, with respect to the sections $c_i$ and $c'_i$. 
The map $C'_i \to J'_j$ of the proof of \boldref{Lem Pic product general} factors as $C'_i \to C_i \to J_j \to J'_j$, so on the Albanese we get
\[
\phi'_{ji} = f_j^* \phi_{ji} f_{i,*}.
\]
Since $f_{i,*}f_i^* \colon J_i \to J_i$ is multiplication by $\deg(f_i)$, we deduce that
\[
\phi'_{i_1\ldots i_m} = \deg(f_{i_2}) \cdots \deg(f_{i_{m-1}}) \cdot f_{i_1}^* \phi_{i_1\ldots i_m} f_{i_m,*}.
\]
Taking $\Q$-vector spaces spanned by these elements proves the first statement. For the final statements, note that the pair $(\iota,\pi) = (\tfrac{1}{\deg(f_i)}f_i^*,f_{i,*})$ as in (\ref{Dia isogeny factor}) realises $J_i$ as an isogeny factor of $J'_i$, and this is an isogeny if $g(C'_i) = g(C_i)$.
\end{proof}
\vskip-\lastskip
Next, we look at how the $E_i(\mathscr L)$ interact with specialisation of endomorphisms.

\begin{Def}\label{Def specialisation}
If $R$ is a DVR with fraction field $K$ and residue field $k$, $S = \Spec R$ is its spectrum, and $T$ is an $S$-scheme satisfying the valuative criterion of properness, then we get a specialisation map
\[
\operatorname{sp} \colon T(K) \cong T(R) \to T(k).
\]
In particular, we may apply this to $T = \PIC_{\mathcal X/S}$ for $\mathcal X \to S$ a smooth proper $S$-scheme with geometrically integral fibres, or to $T = \HOM_S(\mathcal A,\mathcal B)$ where $\mathcal A$ and $\mathcal B$ are abelian schemes over $S$. In the latter case, the specialisation map is an injective group homomorphism (see e.g.~\cite[Cor.~4.3.4]{Thesis}).
\end{Def}


\begin{Lemma}\label{Lem specialisation and composition}
Let $R$ be a DVR, and let $\mathcal C_i$ be smooth projective geometrically integral curves over $\Spec R$ with sections $\sigma_i$. Let $\mathcal X$ be their fibre product. Let $\mathscr L_K \in \Pic(\mathcal C_{i,K})$, and let $\mathscr L_0 \in \Pic(\mathcal C_{i,0})$ be its specialisation. Then for all $i$, we have
\[
\operatorname{sp}(E_i(\mathscr L_K)) = E_i(\mathscr L_0).
\]
If $\mathscr L_0$ corresponds to an isogeny factor $A_0$ of $\Jac_{\mathcal C_{i,0}}$, then $\mathscr L_K$ corresponds to the isogeny factor $\mathcal A_K$ of $\Jac_{\mathcal C_{i,K}}$ for an abelian scheme $\mathcal A$ over $R$ whose special fibre $\mathcal A_0$ is isogenous to $A_0$.
\end{Lemma}

\begin{proof}
Let $\mathscr L_{i,K}$ ($\mathscr L_{i,0}$) and $\phi_{ji,K}$ ($\phi_{ji,0}$) denote the components of $\mathscr L_K$ ($\mathscr L_0$) as in \boldref{Def components of line bundle}. Since specialisation acts componentwise on the right hand side of \boldref{Lem Pic product general}, we get $\operatorname{sp}(\phi_{ji,K}) = \phi_{ji,0}$. We deduce that
\[
\operatorname{sp}(\phi_{i_1\ldots i_m, K}) = \phi_{i_1\ldots i_m, 0}.
\]
Taking $\Q$-vector spaces spanned by these elements proves the first statement.

For the final statement, if $E_i(\mathscr L_0) = p\End^\circ(\Jac_{\mathcal C_{i,0}})p$ for some idempotent $p$, then $p = \operatorname{sp}(q)$ for some $q \in \End^\circ(\Jac_{\mathcal C_{i,K}})$. Since specialisation is injective, we conclude that such $q$ is unique, and that $q$ is an idempotent as well. Similarly, $\psi_K \in \End^\circ(\Jac_{\mathcal C_{i,K}})$ satisfies $q\psi_K = \psi_K = \psi_K q$ if and only if $\psi_0 = \operatorname{sp}(\psi_K)$ satisfies $p\psi_0 = \psi_0 = \psi_0 p$. Thus, we conclude from the first statement that $E_i(\mathscr L_K) = q\End^\circ(\Jac_{\mathcal C_{i,K}})q$.

Let $\mathcal A_K$ be the isogeny factor corresponding to $q$. Then $\mathcal A_K$ has good reduction by N\'eron--Ogg--Shafarevich \cite[Thm.~1]{SerTat}, since $\Jac_{\mathcal C_{i,K}}$ does. Let $\mathcal A$ be the N\'eron model over $\Spec R$. Let $(\iota,\pi)$ correspond to the idempotent $q$ as in \boldref{Def corresponds to isogeny factor}. By the N\'eron property of abelian schemes, $\pi$ extends uniquely to a morphism $\pi \colon \PIC_{\mathcal C_i/R} \to \mathcal A$. Similarly, if $n$ is such that $n \iota \in \Hom(\mathcal A_K, \Jac_{\mathcal C_{i,K}})$, then $n \iota$ extends uniquely to a morphism $\mathcal A \to \PIC_{\mathcal C_i/R}$, which we also denote $n \iota$. The uniqueness statement implies that $\pi_0 \iota_0 = \id$ and $\iota_0 \pi_0 = p$. Therefore $p$ corresponds to the reduction $\mathcal A_0$ of $\mathcal A_K$, hence $\mathcal A_0$ is isogenous to $A_0$.
\end{proof}

\begin{Rmk}
Unlike \boldref{Lem pullback and composition}, there is no converse to the final statement of \boldref{Lem specialisation and composition}. For example, if $\End^\circ(\Jac_{\mathcal C_{i,0}})$ is larger than $\End^\circ(\Jac_{\mathcal C_{i,K}})$ and $\mathscr L_{i,K}$ generates all endomorphisms of $\Jac_{\mathcal C_{i,K}}$, then $\mathscr L_{i,0}$ does not generate all endomorphisms of $\Jac_{\mathcal C_{i,0}}$, and in fact does not correspond to \emph{any} isogeny factor.
\end{Rmk}

Using \boldref{Lem specialisation and composition}, we can use $E_i(\mathscr L)$ as an obstruction to lifting line bundles.

\begin{Prop}\label{Prop no multiple lifts}
Let $C_1,\ldots,C_r$ be smooth projective curves over a field $k$ of characteristic $p > 0$ such that all endomorphisms of $J_1, \ldots, J_r$ are defined over $k$. Let $\mathscr L$ be a line bundle on $X = \prod_i C_i$ that corresponds to a nonzero supersingular isogeny factor $A$ of $\End^\circ(J_i)$ for some $i$ (see \boldref{Def corresponds to isogeny factor}). Then for any DVR $R$ with residue field $k$ and any lifts $\mathcal C_i \to \Spec R$ of the $C_i$, no multiple $\mathscr L^{\otimes m}$ for $m > 0$ can be lifted to $\mathcal X = \prod_i \mathcal C_i$.
\end{Prop}

\begin{proof}
Note that $E_i(\mathscr L^{\otimes m}) = E_i(\mathscr L)$, so we may take $m = 1$. Suppose $\mathcal C_i$ are lifts of the $C_i$ and $\widetilde{\mathscr L}$ is a lift of $\mathscr L$. By \boldref{Lem specialisation and composition}, $\widetilde{\mathscr L}_K$ corresponds to a lift $\mathcal A_K$ of $A$ (up to isogeny). From the equality $\operatorname{sp}(E_i(\mathscr L_K)) = E_i(\mathscr L_0) $, it follows that specialisation $\End^\circ(\mathcal A_K) \stackrel\sim\to \End^\circ(A)$ is an isomorphism. 

But $A$ is supersingular, so by a dimension count it is impossible to lift all its endomorphisms simultaneously (see e.g.~\cite[Cor.~4.3.9]{Thesis}).
\end{proof}

\section{Generation by Rosati dual elements}\label{Sec Rosati}
In \boldref{Lem generate End} below, we give an example of the situation of \boldref{Prop no multiple lifts}. The following slightly more technical result is needed to make an example of minimal dimension in \boldref{thm intro}. The reader who does not care about such matters may skip the proof; see \boldref{Rmk Albert}.

\begin{Thm}\label{Thm Rosati dual elements}
Let $(A,\phi)$ be a polarised supersingular abelian variety of dimension $g \geq 2$ over a field $k$ containing $\bar \F_p$. Then there exists an element $x \in \End^\circ(A)$ such that $x$ and $x^\dagger = \phi^{-1}x\T\phi$ generate $\End^\circ(A)$ as $\Q$-rng. 
\end{Thm}

\begin{proof}
Any supersingular abelian variety over a field containing $\bar \F_p$ is isogenous to $E^g$, where $E$ is a supersingular elliptic curve. Then $D = \End^\circ(E)$ is the quaternion algebra over $\Q$ ramified only at $p$ and $\infty$, and $\End^\circ(A) \cong M_g(D)$. Moreover, when $A$ is supersingular, the Rosati involution on $\End^\circ(A)$ does not depend on the rational polarisation used \cite[Prop.~1.4.2]{Eke}, so we may assume that $\phi$ is the product polarisation. 

Then the Rosati involution on $M_g(D)$ is given by
\begin{align*}
(-)^\dagger \colon M_g(D) &\to M_g(D) \\
(a_{ij}) &\mapsto (a_{ji}^\dagger),
\end{align*}
where $a^\dagger = \operatorname{Trd}(a) - a$ is the Rosati involution on $D = \End^\circ(E)$.

Write $\A(M_g(D))$ for $M_g(D)$ viewed as affine space over $\Q$, and note that the ring operations are given by morphisms of $\Q$-varieties. Then the set $U$ of elements $x \in \A(M_g(D))$ such that $x$ and $x^\dagger$ generate $M_g(D)$ as $\Q$-rng is Zariski open. Indeed, for every subset $W \subseteq \mathbb Z^{\ast 2}\setminus \{e\}$ of size $4g^2$ of nontrivial words, the locus in $\A(M_g(D))$ where $\{w(x,x^\dagger)\ |\ w \in W\}$ generates $M_g(D)$ as $\Q$-vector space is given by the nonvanishing of a certain $4g^2\times 4g^2$ determinant whose coefficients depend on $x$ through the structure coefficients for multiplication and involution. For each $W$ this gives an open set where the $w(x,x^\dagger)$ generate, and $U$ is the union of these open sets over all sets $W\subseteq\Z^{\ast 2}\setminus\{e\}$ of size $4g^2$.

But an open subset $U \subseteq \A^{4g^2}_{\mathbb Q}$ has a $\Q$-point if and only if it is nonempty, i.e.~if and only if it has a $\bar \Q$-point. Thus, it suffices to study $\End^\circ(A)\otimes_\Q \bar \Q$. The algebra $\End^\circ(A) \otimes_\Q \bar \Q$ is isomorphic to $M_{2g}(\bar \Q)$, with involution $(-)^\dagger$ given by
\[
\begin{pmatrix} \begin{pmatrix} a_{11} & b_{11} \\ c_{11} & d_{11} \end{pmatrix} &\!\!\!\! \cdots \!\!\!\!& \begin{pmatrix} a_{1g} & b_{1g} \\ c_{1g} & d_{1g} \end{pmatrix} \\ \vdots &\!\!  \!\!& \vdots \\ \begin{pmatrix} a_{g1} & b_{g1} \\ c_{g1} & d_{g1} \end{pmatrix} &\!\!\!\! \cdots \!\!\!\!& \begin{pmatrix} a_{gg} & b_{gg} \\ c_{gg} & d_{gg} \end{pmatrix} \end{pmatrix}\mapsto \begin{pmatrix} \begin{pmatrix} d_{11} & -b_{11} \\ -c_{11} & a_{11} \end{pmatrix} &\!\!\!\! \cdots \!\!\!\!& \begin{pmatrix} d_{g1} & -b_{g1} \\ -c_{g1} & a_{g1} \end{pmatrix} \\ \vdots &\!\!\!\! \!\!\!\!& \vdots \\ \begin{pmatrix} d_{1g} & -b_{1g} \\ -c_{1g} & a_{1g} \end{pmatrix} &\!\!\!\! \cdots \!\!\!\!& \begin{pmatrix} d_{gg} & -b_{gg} \\ -c_{gg} & a_{gg} \end{pmatrix} \end{pmatrix}.
\]
Now consider the matrix
\[
x = \begin{pmatrix} 0 & 1 & & & \\ & 0 & 1 & & \\ & & & \ddots & \\ & & & & 1 \\ & & & & 0 \end{pmatrix}.
\]
We want to show that the $\bar \Q$-subrng $B \subseteq M_{2g}(\bar \Q)$ generated by $x$ and $x^\dagger$ is $M_{2g}(\bar \Q)$. One easily computes
\begin{align*}
x^{2g-1} &= e_{1,2g}, \\
x^{2g-3} &= e_{1,2g-2} + e_{2,2g-1} + e_{3,2g}, \\
(x^\dagger)^{2g-3} = (x^{2g-3})^\dagger &= - \left( e_{2g-3,2} + e_{2g,1} + e_{2g-1,4} \right).
\end{align*}
Write $a = x^{2g-1}$ and $b = (x^\dagger)^{2g-3}$, which makes sense because $g \geq 2$. Then $ab = -e_{11}$, hence $bab = -e_{2g,1}$. Thus $x - bab$ is the rotation matrix $\rho$ given by $e_i \mapsto e_{i-1}$ for $i > 1$ and $e_1 \mapsto e_{2g}$. Now the matrices $\rho^a e_{11} \rho^b$ for various $a$ and $b$ give all standard basis vectors $e_{ij}$, hence the matrix algebra $M_{2g}(\bar \Q)$ is generated (as $\bar \Q$-rng) by $x - bab$ and $ab$. Thus, $B = M_{2g}(\bar \Q)$.
\end{proof}

\begin{Rmk}
The theorem is false for $g = 1$. Indeed, for any $x \in D$, we have $x^\dagger = \operatorname{Trd}(x) - x$, so in particular $x$ and $x^\dagger$ commute. Therefore, the non-commutative algebra $D$ can never be generated by an element and its Rosati transpose.
\end{Rmk}

Using the theorem, we construct an example of the situation of \boldref{Prop no multiple lifts}.

\begin{Lemma}\label{Lem generate End}
Let $k$ be an extension of $\bar \F_p$, and let $C$ be a supersingular curve over $k$ of genus $g \geq 2$. Let $r \geq 3$, and set $C_i = C$ for all $i \in \{1,\ldots,r\}$. Then there exists a very ample line bundle $\mathscr L$ on $\prod_i C_i$ such that $E_1(\mathscr L) = \End^\circ(J_1)$.
\end{Lemma}

\begin{Rmk}
That is, $\mathscr L$ generates all endomorphisms of $J_1$ (in the sense of \boldref{Def corresponds to isogeny factor}), hence in particular corresponds to a nonzero supersingular isogeny factor. Thus, \boldref{Prop no multiple lifts} implies that no multiple $\mathscr L^{\otimes m}$ for $m > 0$ can be lifted to $\prod \mathcal C_i$, for \emph{any} lifts $\mathcal C_i$ of $C_i$.
\end{Rmk}

\begin{proof}[Proof of Lemma.]
By \boldref{Thm Rosati dual elements}, there exists $x \in \End^\circ(J_1)$ such that $x$ and $x^\dagger = x\T$ generate $\End^\circ(J_1)$ as $\Q$-rng. Now set $\phi_{21} = x$, and $\phi_{31} = \phi_{32} = 1$. Then the maps
\begin{align*}
\phi_{1321} = \phi_{31}\T\phi_{32}\phi_{21} \colon J_1 &\to J_1\\
\phi_{1231} = \phi_{21}\T\phi_{32}\T\phi_{31} \colon J_1 &\to J_1
\end{align*}
are given by $x$ and $x\T$ respectively. In \boldref{Pic noncommutative}, this corresponds to going around the following loops (where all unmarked arrows are the identity):
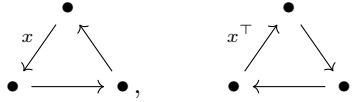
\begin{equation*}
\begin{tikzcd}[column sep=.7em]
& \bullet \ar{ld}[swap]{x} & & & & & & \bullet \ar{rd} & \\
\bullet \ar{rr} & & \bullet\punct{,} \ar{lu} & & & & \bullet \ar{ru}{x\T} & & \bullet\punct{.}\ar{ll}
\end{tikzcd}
\end{equation*}
If $P \in C(k)$ is a rational point, then the line bundle $\mathcal O(P)^{\boxtimes r}$ is very ample. Hence, for $d \gg 0$, the line bundle
$$\mathscr L = \left(\big(\mathcal O_{C_i}(dP)\big)_i, \big(\phi_{ji}\big)_{i < j}\right)$$
is very ample and satisfies $E_1(\mathscr L) = \End^\circ(J_1)$.
\end{proof}

\begin{Rmk}\label{Rmk Albert}
For $r \geq 4$ we do not need to use \boldref{Thm Rosati dual elements}. Indeed, by Albert's theorem on generation of separable algebras \cite{Alb} there exist $x, y \in \End^\circ(J_1)$ that generate it as $\Q$-algebra (see also \cite[Thm.~7.2.1]{Thesis} for an elementary geometric proof analogous to our proof of \boldref{Thm Rosati dual elements} above). Hence the elements $1$, $x$, and $y$ generate $\End^\circ(J_1)$ as $\Q$-rng.

Then we can run the argument of \boldref{Lem generate End} using $\phi_{42} = x$, $\phi_{43} = y$, and all other $\phi_{ji}$ equal to $1$. The loops $\phi_{1421} = x$, $\phi_{1431} = y$, and $\phi_{1321} = 1$ then show $E_1(\mathscr L) = \End^\circ(J_1)$. 

This corresponds to going around the following loops in \boldref{Pic noncommutative} (where again all unmarked arrows are the identity):
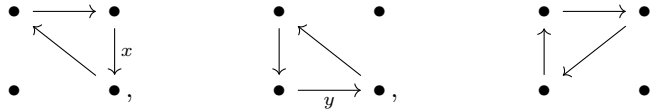
\begin{equation*}
\begin{tikzcd}
\bullet \ar{r} & \bullet \ar{d}{x} & & \bullet \ar{d} & \bullet & & \bullet \ar{r} & \bullet \ar{ld} \\
\bullet & \bullet\punct{,} \ar{ul} & & \bullet \ar{r}[swap]{y} & \bullet\punct{,} \ar{ul} & & \bullet \ar{u} & \bullet\punct{.}
\end{tikzcd}
\end{equation*}
\end{Rmk}

\section{Stably irreducible divisors}\label{Sec stably irreducible}
We introduce the following property that plays a role in the proof of \boldref{thm intro}, as explained at the end of the \hyperref[Sec intro]{Introduction}.

\begin{Def}\label{Def stably irreducible}
Let $k$ be a field, and let $C_1, \ldots, C_r$ be smooth projective curves over $k$. Then an effective divisor $D \subseteq \prod_{i=1}^r C_i$ is \emph{stably irreducible} 
if for all finite coverings $f_i \colon C'_i \to C_i$ of the $C_i$ by smooth projective curves $C'_i$, the inverse image $D' \subseteq \prod_i C'_i$ of $D$ under $f \colon \prod_i C'_i \to \prod_i C_i$ is geometrically irreducible. In particular, $D$ itself is geometrically irreducible.
\end{Def}

We will show that a sufficiently general divisor $D$ satisfies this property; see \boldref{Prop stably irreducible}. The local computation we use is the following.

\begin{Lemma}\label{Lem stably irreducible}
Let $r \geq 2$, let $C_1, \ldots, C_r$ be smooth projective curves over a field $k$, and let $D \subseteq \prod_{i=1}^r C_i$ be an ample effective divisor. Assume that all of the following hold:
\begin{enumerate}
\item $D$ is geometrically normal;\label{Item normal}
\item $D \cap \pi_i^{-1}(x_i)$ is generically smooth for all $i$ and all $x_i \in C_i$;\label{Item single fibre}
\item $D$ does not contain $\pi_i^{-1}(x_i) \cap \pi_j^{-1}(x_j)$ for any closed points $x_i \in C_i$ and $x_j \in C_j$ for $i \neq j$.\label{Item double fibre}
\end{enumerate}
Then $D$ is stably irreducible.
\end{Lemma}

\begin{proof}
Since all statements are geometric, we may assume $k$ is algebraically closed. Let $f_i \colon C'_i \to C_i$ be finite coverings by smooth projective curves. 
If $f_i$ is purely inseparable, then it is a universal homeomorphism. This does not affect irreducibility, so we only have to treat the case that the $f_i$ are separable, i.e.~generically \'etale.

The inverse image $D' = f^{-1}(D)$ is ample since $D$ is \cite[Cor.~6.6.3]{EGA2}, hence $D'$ is connected since $r \geq 2$ \cite[Cor.~III.7.9]{Hart}. Since $D'$ is a divisor in a regular scheme, it is Cohen--Macaulay \cite[Tag \href{http://stacks.math.columbia.edu/tag/02JN}{02JN}]{Stacks}. We will show that the assumptions on $D$ imply that $D'$ is regular in codimension $1$. Then Serre's criterion implies that $D'$ is normal \cite[Thm.~5.8.6]{EGA4II}. Then $D'$ is integral, since it is normal and connected \cite[5.13.5]{EGA4II}.

Now let $x' \in D'$ be a point of codimension $1$, and consider the image $x'_i$ of $x'$ in $C'_i$. Let $\eta'_i$ be the generic point of $C'_i$. Let $x$, $x_i$, and $\eta_i$ be the images of $x'$ in $X$, of $x'_i$ in $C_i$, and of $\eta'_i$ in $C_i$ respectively. Consider the set
\[
I = \left\{i \in \{1,\ldots,r\}\ \big| \ x'_i \neq \eta'_i\right\} = \left\{i\ \big| \ x_i \neq \eta_i\right\}
\]
of $i \in \{1,\ldots,r\}$ such that $x'$ does not dominate the factor $C'_i$, i.e.~$x'$ lies in a closed fibre of the projection $X' \to C'_i$.

If $|I| > 2$, then $\overline{\{x'\}} \subseteq \pi_i^{-1}(x_i) \cap \pi_j^{-1}(x_j) \cap \pi_k^{-1}(x_k)$ for some $i,j,k \in \{1,\ldots,r\}$ pairwise distinct, contradicting the fact that $x'$ has codimension $1$ in $D'$. If $|I| = 2$, then $x'$ is the generic point of $\pi_i^{-1}(x'_i) \cap \pi_j^{-1}(x'_j)$ for $i \neq j$, hence $D$ contains $\pi_i^{-1}(x_i) \cap \pi_j^{-1}(x_j)$, contradicting assumption \ref{Item double fibre}. Hence, $|I| \leq 1$.

If $|I| = 0$, then $x$ maps to $\eta_i$ for each $i$, hence $\mathcal O_{D,x}$ contains the fields $\kappa(\eta_i)$ for all $i$. Since $f_i$ is separable, the field extension $\mathcal O_{C_i,\eta_i} \to \mathcal O_{C'_i,\eta'_i}$ is \'etale. Hence, $x$ is in the \'etale locus of $D' \to D$. But $\mathcal O_{D,x}$ is regular by assumption \ref{Item normal}, so the same goes for $\mathcal O_{D',x'}$ \cite[Prop.~17.5.8]{EGA4IV}.

Finally, if $|I| = 1$, then $x$ is the generic point of a component of $D \cap \pi_i^{-1}(x_i)$ for some $i$, and similarly for $x'$. As in the case $|I| = 0$, the extensions $C'_j \to C_j$ for $j \neq i$ do not affect normality at $x$, so we may assume that $C'_j = C_j$ for $j \neq i$. Then the natural map $D' \cap \pi_i^{-1}(x'_i) \to D \cap \pi_i^{-1}(x_i)$ is an isomorphism, since $\pi_i^{-1}(x'_i) = \prod_{j \neq i} C'_i \stackrel\sim\rA \prod_{j \neq i} C_j = \pi_i^{-1}(x_i)$ only sees the curves $C'_j = C_j$ for $j \neq i$.

Consider the local homomorphism $\mathcal O_{C'_i,x'_i} \to \mathcal O_{D',x'}$. It is flat, since \ref{Item double fibre} implies that every irreducible component of $D'$ dominates $C'_i$. Moreover, the fibre $\mathcal O_{D', x'}/\mathfrak m_{x'_i}\mathcal O_{D', x'}$ is a field, since $D' \cap \pi_i^{-1}(x'_i) = D \cap \pi_i^{-1}(x_i)$ is generically smooth by assumption \ref{Item single fibre}. Since $\mathcal O_{C'_i, x'_i}$ is regular and $\mathcal O_{C'_i,x'_i} \to \mathcal O_{D',x'}$ is flat and local, we conclude that $\mathcal O_{D',x'}$ is regular \cite[Prop.~6.5.1(ii)]{EGA4II}.
\end{proof}

\begin{Prop}\label{Prop stably irreducible}
Let $r \geq 3$, and let $C_1, \ldots, C_r$ be smooth projective curves over $k$. Let $H$ be an ample divisor on $\prod_{i=1}^r C_i$. Then there exists $n_0 \in \Z_{>0}$ such that for all $n \geq n_0$, a general divisor $D \in |nH|$ is stably irreducible.
\end{Prop}

\begin{proof}
There exists $n_0$ such that for all $n \geq n_0$, the divisor $nH$ is very ample. By the usual Bertini smoothness theorem, a general $D \in |nH|$ is smooth, so in particular geometrically normal. Increasing $n_0$ if necessary, for a general $D$ all fibres $D \cap \pi_i^{-1}(x_i)$ are generically smooth (see e.g.~\cite[Lem.~3.1.2]{Thesis}).  

Similarly, we may avoid any finite type family of positive-dimensional subvarieties (see e.g.~\cite[Lem.~3.1.3]{Thesis}), so a general $D$ does not contain any double fibre $\pi_i^{-1}(x_i) \cap \pi_j^{-1}(x_j)$ (these are positive-dimensional since $r \geq 3$). Then \boldref{Lem stably irreducible} shows that these $D$ are stably irreducible.
\end{proof}

\begin{Rmk}
On the other hand, for $r \leq 2$ no effective divisor $D \subseteq \prod_i C_i$ is stably irreducible. This is obvious if $r \leq 1$ and for $r = 2$ if $D$ is pulled back from either curve. For `diagonal' divisors $D \subseteq C_1 \times C_2$, we can first apply a cover to $C_1$ to make its degree in $C_1 \times \Spec K(C_2)$ larger than $1$. Then it picks up an $L$-rational point after a finite extension $L = K(C'_2)$ of $K(C_2)$, hence it becomes reducible in $C'_1 \times C'_2$.
\end{Rmk}

\begin{Ex}
The conclusion of \boldref{Prop stably irreducible} is not true for \emph{all} smooth divisors $D \in |nH|$. For example, let $r = 3$, $C_i = \P^1$ with coordinates $[x_i:y_i]$, and let $D$ be given by $x_1 x_2 x_3 - y_1 y_2 y_3 \in H^0((\P^1)^3, \mathcal O(1)^{\boxtimes 3})$. Consider the affine charts associated with inverting one of $\{x_i, y_i\}$ for each $i$. Then the local equations are $xyz - 1$ and $xy - z$, both of which define a smooth surface.

However, if we take the covers given by $C'_i = \P^1$ with map $C'_i \to C_i$ given by $[x_i:y_i] \mapsto [x_i^2:y_i^2]$, then $D'$ splits as $V(x_1 x_2 x_3 - y_1 y_2 y_3) \cup V(x_1 x_2 x_3 + y_1 y_2 y_3)$. So even when $D$ is smooth (in \emph{arbitrary} characteristic), it is not always stably irreducible. This $D$ violates assumption \ref{Item double fibre} of \boldref{Lem stably irreducible} because it contains $\pi_1^{-1}([0:1]) \cap \pi_2^{-1}([1:0])$.
\end{Ex}

\section{Main construction}\label{Sec main construction}
For every prime $p$, we construct a smooth projective surface $X$ over $\bar \F_p$ with the property that no smooth proper variety $Y$ dominating $X$ can be lifted to characteristic $0$; see \boldref{Con main} and \boldref{Thm main} below.

\begin{Con}\label{Con main}
Let $p$ be a prime, let $r \geq 3$ be an integer, and let $k$ be an algebraically closed field of characteristic $p$. Let $C_1 = \ldots = C_r = C$ be a supersingular curve over $k$ of genus $g \geq 2$. For example, the Fermat curve $x^{q+1} + y^{q+1} + z^{q+1} = 0$ is supersingular if $q$ is a power of $p$ \cite[Lem.~2.9]{KatShi}. Alternatively, a smooth member of Moret-Bailly's family \cite{MB, MB2} is a supersingular curve of genus $2$. Both examples are defined over $\bar \F_p$.

By \boldref{Lem generate End} (or \boldref{Rmk Albert} if $r \geq 4$), there exists a very ample line bundle $\mathscr L$ on $\prod_i C_i$ that generates all endomorphisms of $J_1$ in the sense of \boldref{Def corresponds to isogeny factor}. 

Finally, we define $X \subseteq \prod_i C_i$ as a smooth divisor in $|n\mathscr L|$ for $n \gg 0$ that is stably irreducible (see \boldref{Def stably irreducible}). Such a divisor exists by \boldref{Prop stably irreducible} and the usual Bertini smoothness theorem.
\end{Con}

The following result will be useful in the proof.

\begin{Lemma}\label{Lem irreducible inverse image}
Let $f \colon X \to Y$ be a finite flat morphism of finite type $k$-schemes. Let $V \subseteq X$ be an integral subscheme, and let $W = f(V)$ be its image. If $f^{-1}(W)$ is irreducible, then $f^*[W] = d \cdot [V]$ for some $d \in \Z_{> 0}$.
\end{Lemma}

\begin{proof}
Note that $W$ is irreducible since $V$ is. Since specialisations lift along finite morphisms, we have $\dim(V) = \dim(W) = \dim(f^{-1}(W))$. Hence, $V$ is a component of $f^{-1}(W)$. Since $f^{-1}(W)$ is irreducible, we conclude that $V = f^{-1}(W)$ holds set-theoretically. Therefore, $f^*[W]$ is a multiple of $[V]$.
\end{proof}

\begin{Thm}\label{Thm main}
Let $X$ be as in \boldref{Con main}. If $k \subseteq k'$ is a field extension and $Y$ is a smooth proper $k'$-variety with a dominant rational map $Y \dashrightarrow X \times_k k'$, then $Y$ cannot be lifted to characteristic $0$.
\end{Thm}

\begin{Rmk}\label{Rmk surface}
Since $X$ is a divisor in a product of $r \geq 3$ curves, we get examples in every dimension $\geq 2$. Of course, if $X$ is an example of dimension $d$ and $Z$ is any $m$-dimensional smooth projective variety, then $X \times Z$ is an example of dimension $d+m$. 

Since curves are unobstructed, the result in dimension $2$ is the best possible. If $p \geq 5$, then the Bombieri--Mumford classification \cite{BM1, BM2, BM3} together with existing liftability results in the literature \cite{DelLift}, \cite{MumLift}, \cite{NO}, \cite{Sei} imply that every smooth projective surface $X$ of Kodaira dimension $\kappa(X) \leq 1$ can be dominated by a liftable surface \cite[Thm.~6.3.1]{Thesis}. Therefore, our surface of general type is the `easiest' example possible.
\end{Rmk}

\begin{proof}[Proof of Theorem.]
We may replace $k'$ by $\overline{k'}$ and then replace $k$ by $k'$. This does not change the supersingularity of the $C_i$, the generation of all endomorphisms of $J_1$, or the stable irreducibility of $X$. This reduces us to the case $k = k'$.

Any rational map $Y \dashrightarrow C$ to a curve of genus $\geq 1$ can be extended to a morphism $Y \to C_i$ (see e.g.~\cite[Cor.~4.1.4]{Thesis}). In particular, the rational maps $\phi_i \colon Y \dashrightarrow X \to C_i$ extend, hence so does the map
\[
\phi \colon Y \to X \subseteq \prod_{i=1}^r C_i.
\]
Now assume $R$ is a DVR with residue field $k$ and $\mathcal Y$ is a lift of $Y$ over $\Spec R$. Consider the Stein factorisation
\begin{equation}
Y \stackrel{\phi'_i}\rA C'_i \stackrel{f_i}\rA C_i\label{Dia Stein factorisation}
\end{equation}
of the maps $\phi_i \colon Y \to C_i$. 
By \boldref{Thm lift morphism to curve}, after possibly extending $R$ there exist smooth proper curves $\mathcal C''_i$ over $R$, morphisms $\phi'' \colon \mathcal Y \to \mathcal C''_i$, and commutative diagrams
\begin{equation}
\begin{tikzcd}[column sep=.6em]
 & \mathcal Y_0 \ar{ld}[swap]{\phi'_i}\ar{rd}{\phi''_{i,0}} & \\
C'_i \ar{rr}[swap]{F_i} & & \mathcal C''_{i,0}\punct{,}
\end{tikzcd}\label{Dia triangle lift}
\end{equation}
where $F_i$ is purely inseparable (hence a power of Frobenius). Write $Z = \prod_i C_i$, $Z' = \prod_i C'_i$, and $\mathcal Z'' = \prod_i \mathcal C''_i$. The product over all $i$ of (\ref{Dia Stein factorisation}) and (\ref{Dia triangle lift}) gives the commutative diagram
\begin{equation*}
\begin{tikzcd}[row sep=.25cm,column sep=1.9em]
 & Y \ar{dd}[yshift=-.3em]{\phi'}\ar{lddd}[swap]{\phi}\ar{rddd}{\phi''_0} & \\
 & \  & \\
 & Z' \ar{ld}{f} \ar{rd}[swap]{F} & \\
Z & & \mathcal Z''_0\punct{.}
\end{tikzcd}
\end{equation*}
Then the image $X' = \phi'(Y) \subseteq Z'$ satisfies $f(X') = X \subseteq Z$. The preimage $f^{-1}(X) \subseteq Z'$ is irreducible since $X$ is stably irreducible by \boldref{Con main}. Thus, \boldref{Lem irreducible inverse image} implies that
\begin{equation}
f^*[X] = a \cdot [X']\label{Eq a}
\end{equation}
for some $a \in \Z_{>0}$. 
Since $F \colon Z' \to \mathcal Z''_0$ is radicial, another application of \boldref{Lem irreducible inverse image} shows that the image $X'' = \phi''_0(Y) \subseteq \mathcal Z''_0$ satisfies
\begin{equation}
F^*[X''] = b \cdot [X']\label{Eq b}
\end{equation}
for some $b \in \Z_{>0}$. 
Finally, let $\mathcal X'' = \phi''(\mathcal Y) \subseteq \mathcal Z''$ be the scheme-theoretic image of $\phi'' \colon \mathcal Y \to \mathcal Z''$. Then $\mathcal X''$ is flat over $R$ since the image factorisation $\mathcal O_{\mathcal Z''} \twoheadrightarrow \mathcal O_{\mathcal X''} \hookrightarrow \phi''_* \mathcal O_{\mathcal Y}$ realises $\mathcal O_{\mathcal X''}$ as a subsheaf of the $R$-torsion-free sheaf $\phi''_* \mathcal O_{\mathcal Y}$. Hence, $\mathcal X''$ is a lift of its special fibre $\mathcal X''_0$ as a divisor. Since $\mathcal X''_0$ agrees set-theoretically with the reduced divisor $X'' = \phi''_0(Y)$, we conclude that
\begin{equation}
[\mathcal X''_0] = c \cdot [X'']\label{Eq c}
\end{equation}
for some $c \in \Z_{>0}$. 
Combining (\ref{Eq a}), (\ref{Eq b}), and (\ref{Eq c}), we conclude that
\begin{equation}
bc \cdot f^*[X] = a \cdot F^* [\mathcal X''_0].\label{Eq d}
\end{equation}
But $[X] = [\mathscr L]$ is given by a line bundle $\mathscr L$ that generates all endomorphisms of the supersingular abelian variety $J_1$, by \boldref{Con main}. Hence, $f^*[\mathscr L]$ corresponds to the supersingular isogeny factor $J_1$ of $J'_1$ by \boldref{Lem pullback and composition}. Hence the same holds for $[\mathcal X''_0]$ by (\ref{Eq d}) and \boldref{Lem pullback and composition}. Finally, \boldref{Prop no multiple lifts} then shows that $\mathcal O_{\mathcal Z''_0}(\mathcal X''_0)$ does not lift to a line bundle on $\mathcal Z''$. This contradicts the fact that $\mathcal X''$ is a lift of $\mathcal X''_0$ as a divisor.
\end{proof}

\begin{Rmk}\label{Rmk Zariski open}
The proofs of \boldref{Thm Rosati dual elements} and \boldref{Lem generate End} show that the set of $\mathscr L \in \Pic(\prod_i C_i)$ that generate all endomorphisms of $J_1$ (as in \boldref{Def corresponds to isogeny factor}) form the integral points of the intersection of the ample cone with a Zariski open subset of $\NS(\prod_i C_i)\otimes \Q$. Similarly, the set of stably irreducible divisors $X \in |\mathscr L^{\otimes n}|$ contains a Zariski open (which is nonempty for $n \gg 0$) by the proof of \boldref{Prop stably irreducible}. This shows that `most' divisors in $\prod_i C_i$ give counterexamples to the main question.
\end{Rmk}

\begin{Rmk}\label{Rmk formal}
Our methods do not address the weaker question of dominating $X$ by a smooth proper variety $Y$ that admits a \emph{formal} lift to characteristic $0$. Similarly, our methods do not answer Bhatt's question \cite[Rmk.~5.5.5]{BhattThesis} whether every smooth projective variety $X$ can be dominated by a smooth proper variety $Y$ that admits a lift to the length $2$ Witt vectors $W_2(k)$.
\end{Rmk}

\phantomsection
\printbibliography
\end{document}